\documentclass[11pt,letterpaper]{article}

\usepackage[utf8]{inputenc}
\usepackage[T1]{fontenc}
\usepackage[english]{babel}
\usepackage[hyphens]{url}
\usepackage{authblk}
\usepackage{booktabs}
\usepackage[colorlinks=true, linkcolor=red, urlcolor=blue, citecolor=blue, anchorcolor=blue,backref=page]{hyperref}
\usepackage{amsfonts}
\usepackage{nicefrac}
\usepackage{microtype}
\usepackage{amsmath, amsthm, amssymb}
\usepackage{fancybox}
\usepackage{graphicx}
\usepackage{float}
\usepackage{algorithm}
\usepackage{algpseudocode}
\usepackage{color}
\usepackage{tikz}
\usepackage{array}
\usepackage{subcaption}
\usepackage[textsize=scriptsize]{todonotes}
\usepackage{enumitem}
\usepackage{wrapfig}
\usepackage{enumitem}
\usepackage{etoolbox}
\usepackage{diagbox}
\usepackage{comment}
\usepackage{makecell}
\usepackage{multirow}
\usepackage{bbm}
\usepackage[in]{fullpage}
\usepackage{multicol}
\usepackage{mathpazo}

\newtheoremstyle{slplain}
  {.4\baselineskip\@plus.1\baselineskip\@minus.1\baselineskip}
  {.3\baselineskip\@plus.1\baselineskip\@minus.1\baselineskip}
  {\itshape}
  {}
  {\bfseries}
  {.}
  { }
  {}
\theoremstyle{slplain} 

\newtheorem*{definition*}{Definition}
\newtheorem*{theorem*}{Theorem}
\newtheorem{theorem}{Theorem}[section]
\newtheorem{lemma}[theorem]{Lemma}

\newtheorem{claim}[theorem]{Claim}
\newtheorem{corollary}[theorem]{Corollary}
\newtheorem{definition}[theorem]{Definition}

\makeatletter
\newtheorem*{rep@theorem}{\rep@title}
\newcommand{\newreptheorem}[2]{%
\newenvironment{rep#1}[1]{%
 \def\rep@title{#2 \ref{##1}}%
 \begin{rep@theorem}}%
 {\end{rep@theorem}}}
\makeatother

\newreptheorem{theorem}{Theorem}
\newreptheorem{lemma}{Lemma}
\newreptheorem{claim}{Claim}
\newreptheorem{corollary}{Corollary}
\newreptheorem{proposition}{Proposition}

\theoremstyle{definition}

\newtheorem{remark}[theorem]{Remark}

\theoremstyle{plain} 

\numberwithin{equation}{section}

\newtheoremstyle{etplain}
  {.0\baselineskip\@plus.1\baselineskip\@minus.1\baselineskip}
  {.0\baselineskip\@plus.1\baselineskip\@minus.1\baselineskip}
  {\itshape}
  {}
  {\bfseries}
  {.}
  { }
  {}

\newcommand{\R}{\mathbb{R}}

\newcommand\mc[1]{\mathcal{#1}}
\renewcommand\bar\overline

\DeclareMathOperator*{\argmin}{arg\,min}

\newcolumntype{C}[1]{>{\centering\let\newline\\\arraybackslash\hspace{0pt}}m{#1}}

\DeclareMathOperator{\E}{\mathbb{E}}

\DeclareMathOperator{\trace}{tr}

\newcommand{\calE}{\ensuremath{\mathcal{E}}}
\newcommand{\calF}{\ensuremath{\mathcal{F}}}

\newcommand{\calK}{\ensuremath{\mathcal{K}}}

\newcommand{\calN}{\ensuremath{\mathcal{N}}}
\newcommand{\calO}{\ensuremath{\mathcal{O}}}

\newcommand{\calX}{\ensuremath{\mathcal{X}}}

\newcommand{\bbP}{\ensuremath{\mathbb{P}}}

\newcommand{\bbR}{\ensuremath{\mathbb{R}}}
\newcommand{\bbS}{\ensuremath{\mathbb{S}}}

\newcommand{\bbZ}{\ensuremath{\mathbb{Z}}}

\def\nd/{\textsuperscript{nd}}
\def\rd/{\textsuperscript{rd}}
\def\th/{\textsuperscript{th}}

\makeatletter
\def\nnil{\nil}
\newcounter{prob}

\newcounter{dual}

\makeatother

\newenvironment{prob*}{%
	\csname equation*\endcsname%
	\aligned%
}{%
	\endaligned%
	\csname endequation*\endcsname%
}

\usepackage{tcolorbox}
\tcbuselibrary{skins,breakable}

\title{Model-Free Learning for the Linear Quadratic Regulator over Rate-Limited Channels}
\author{Lintao Ye$^\dagger$, Aritra Mitra$^\dagger$, and Vijay Gupta
\thanks{The first two authors contributed equally to this work. L. Ye is with the School of Artificial Intelligence and Automation, Huazhong University of Science and Technology Email: {\tt yelintao93@hust.edu.cn}. A. Mitra is with the Department of Electrical and Computer Engineering,  North Carolina State University Email: {\tt amitra2@ncsu.edu}. V. Gupta is with The Elmore Family School of Electrical and Computer Engineering, Purdue University Email: {\tt gupta869@purdue.edu}. An initial version of this paper was presented at the L4DC 2024 conference \cite{mitra2024towards}.}} 

\date{\today}

\begin{document}
\maketitle
\begin{abstract}
Consider a linear quadratic regulator (LQR) problem being solved in a model-free manner using the policy gradient approach. If the gradient of the quadratic cost is being transmitted across a rate-limited channel, both the convergence and the rate of convergence of the resulting controller may be affected by the bit-rate permitted by the channel. We first pose this problem in a communication-constrained optimization framework and propose a new adaptive quantization algorithm titled Adaptively Quantized Gradient Descent (\texttt{AQGD}). This algorithm guarantees exponentially fast convergence to the globally optimal policy, with \textit{no deterioration of the exponent relative to the unquantized setting}, above a certain finite threshold bit-rate allowed by the communication channel. We then propose a variant of \texttt{AQGD} that provides similar performance guarantees when applied to solve the model-free LQR problem. Our approach reveals the benefits of adaptive quantization in preserving fast linear convergence rates, and, as such, may be of independent interest to the literature on compressed optimization. Our work also marks a first step towards a more general bridge between the fields of  model-free control design and networked control systems.
\end{abstract}

\section{Introduction}
\label{sec:Intro}
Reinforcement learning (RL) to solve classical control problems such as the linear quadratic regulator (LQR), the $\mathcal{H}_{2}$ or the $\mathcal{H}_{\infty}$ control problems, is now well-understood. Both the convergence properties of various algorithms and characterization of the convergence rates through a non-asymptotic analysis have been considered. For model-based RL, which requires the construction of an empirical model of the process, we can point to works such as~\cite{tsiamis,ye2022sample,mania2019certainty,simchowitz2020naive}. For model-free approaches that do not involve the construction of such a model, example works include~\cite{hu,zhang2021policy,zhao2023global}. In this paper, we are interested in a popular model-free algorithm -- policy gradient -- applied to the linear quadratic regulator (LQR) problem~\cite{anderson}. This algorithm has been studied extensively in the LQR context. The authors in~\cite{fazel} showed that despite the non-convexity of the optimization landscape, 
 model-free policy gradient algorithms can guarantee convergence to the globally optimal policy. Furthermore, given access to exact policy gradients, this convergence is exponentially fast.

We consider the performance of policy gradient algorithms being used to solve an LQR problem when a rate-limited channel is present in the loop. The problem of robustness of policy gradient algorithms (or reinforcement learning (RL) algorithms more generally) to communication-induced distortions introduced if the transmission of either the gradient or the policy occurs over realistic communication channels has received only limited attention; see, for instance, the work on event-triggered RL in~\cite{chen2021, gatsis2022}, on rate-limited and noisy channels for stochastic bandit problems in~\cite{hanna, pase, mitra2023linear}, on policy evaluation using compressed temporal difference learning in~\cite{TDEF,dalFedTD}), and more recently, on the effect of delays in stochastic approximation~\cite{adibi2024stochastic} and policy gradient algorithms for the LQR problem in~\cite{sha2024asynchronous, toso2024asynchronous}. On the other hand, the impact of realistic communication channels present in a control loop has been studied extensively under the rubric of networked control systems. The simplest formulation in this area considers a plant being controlled by a remote controller that either receives measurements from a sensor across a communication channel or transmits the control input to an actuator over a channel. The setting that we consider, when the gradient of the cost function is transmitted across a communication channel, corresponds to the channel between the sensor and the controller. In this setting, the abstraction of the communication channel as a rate-limited channel or a `bit pipe' that can transmit a finite number of bits per channel use has been studied for at least two decades now. The so-called data-rate theorem states that if the open loop plant is a linear time-invariant (LTI) system, there is a minimal bit-rate - characterized by the sum of the logarithms of the magnitudes of the open loop unstable eigenvalues of the plant - that must be supported by the channel in order for an encoder-decoder and a controller to exist such that the plant can be stabilized~\cite{tatikonda,nair2004stabilizability}. This result has been extended in various directions, including consideration of stochastic or linear switched plants, combining the channel model with other effects such as a packet drop, consideration of event-triggered communication schemes, and so on. We point the reader to works such as~\cite{nair2007feedback,minero2009data,tallapragada2015event,martins2006finite} and the references therein.
 
However, classically, this line of work assumes and exploits the knowledge of the model of the plant being controlled in the encoder-decoder and the controller design. Some limited work~\cite{okano2014minimum} has been done to relax this assumption; however, there is little to no theoretical understanding of networked control systems in the absence of such an assumption. As RL-based control becomes more popular and well-understood, it becomes natural to seek to understand how to remove this assumption. From the point of view of networked control, this would require the design of encoders and decoders for the communication channels and that of the controllers when the plant model is either simultaneously being learned (in model-based RL) or is never explicitly constructed. From the point of view of RL-control, this would require the understanding of the impact of communication channel effects on traditional RL algorithms, and the design of measures to counteract this impact.  In this paper, we take the first step in this direction. While the modeling choices we make (LTI plant, rate-limited channel, and policy gradient based controller design) may seem specific, they already begin to show some of the intricacies that arise when we seek to bridge these two areas. The ultimate goal of this work is to connect control, communication, and learning by initiating a study of model-free control under communication constraints.  

\begin{figure}[t]
\centering
  \includegraphics[width=0.5\linewidth]{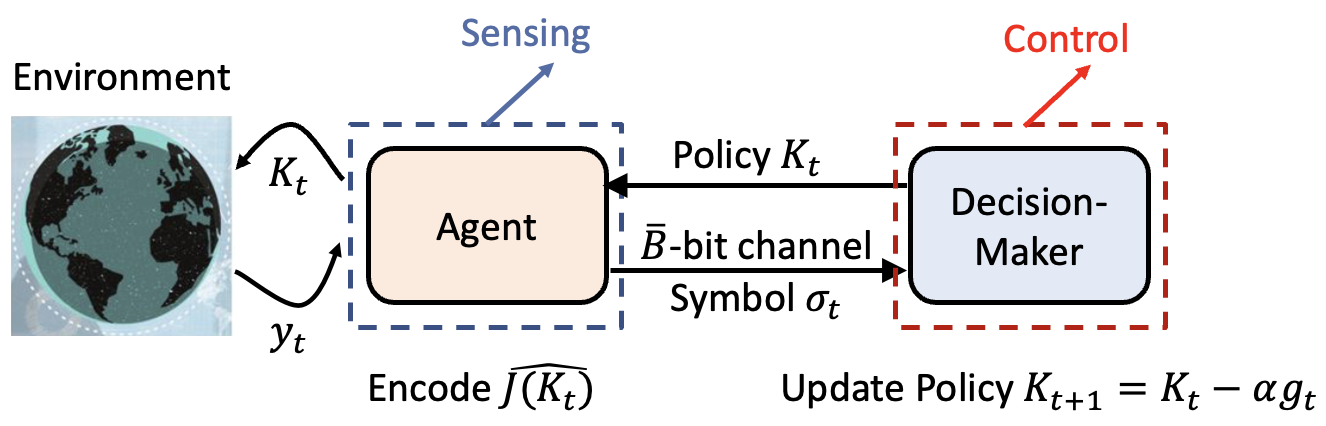}
  \caption{Communication-constrained policy optimization for model-free LQR. At each iteration $t$, the decision-maker sends the current policy $K_t$ to an agent over a noiseless channel of infinite capacity. The agent evaluates and encodes the noisy policy gradient $\widehat{\nabla J(K_t)}$  using $\bar{B}$ bits, and transmits the encoded symbol $\sigma_t$ to the decision-maker over a noiseless rate-limited channel. The decision-maker updates the policy based on the decoded policy gradient $g_t$.}
  \label{fig:Setup}
\end{figure} 

Specifically, in this paper, we consider the setting\footnote{A more careful problem formulation is provided in Section~\ref{sec:LQR setup}, where the notation in the figure is also formally defined.} shown in Figure~\ref{fig:Setup}. Consider a remote agent that observes a plant and transmits information to a controller server across a noiseless channel. The channel is abstracted as a bit-pipe, meaning that $\bar{B}$ bits can be transmitted in a noiseless fashion across it per channel use. The agent has two functions. First, it executes the control policy relayed to it by the server. Notice that the control policy is not communicated across a bit-constrained communication channel to the agent. Such presence of a communication channel only in one of the sensor-controller or controller-actuator paths is standard in networked control systems literature and can model the case where, e.g., the controller has ample transmission power while the sensor is limited by battery constraints. Second, the agent observes the performance of the plant when this control policy is implemented, calculates the observed noisy gradient of the cost function being optimized based on the system trajectory data, and transmits an encoded version of the gradient using $\bar{B}$ bits across the bit-constrained communication channel. The server receives this encoded gradient, decodes it, updates the control policy through a policy gradient type algorithm, and communicates this updated policy to the agent. We assume that the plant being controlled is an \textit{unknown} LTI system and the cost that is sought to be minimized is a quadratic cost. Thus, the collective objective of the agent and the server is to solve an LQR problem based on a model-free RL algorithm.

We comment here that the classical results in networked control systems literature in such a setup seek merely to stabilize the closed-loop system (e.g.,~\cite{tatikonda}), and the optimal control design in the sense of LQR design is an open problem~\cite{nair2007feedback}. Further, we emphasize that while the version considered here is only a first step towards realizing a full theory of model-free control/reinforcement learning in multi-agent networked systems, where communication plays a key role~\cite{linMARL, shinMARL, wang2023model}, it is interesting in its own right parallel to the early developments of networked control systems literature. With such a setup, we seek to design the encoder, decoder, and RL algorithm that can guarantee asymptotic convergence to the globally optimal policy and to understand the loss in performance if that is not possible. In addition, we seek to understand the finite-sample performance of the design in terms of convergence rate, with a particular eye on characterizing any loss in the rate of convergence with a finite value of $\bar{B}$ relative to when the channel has infinite capacity (i.e., when $\bar{B}=\infty$). We summarize our contributions as follows. 

$\bullet$ {\bf Novel Quantized Gradient-Descent Scheme with Linear Convergence Rate.} In Section~\ref{sec:AQGD}, we show that the problem considered is related to the literature on the general problem of quantization in optimization~\cite{gandikota, kostina, mayekar} using a single-worker single-server framework (see Fig.~\ref{fig:Model}). Using this connection, we design a new quantized gradient descent algorithm that carefully exploits the smoothness properties of the objective function and utilizes the allowed bits better by encoding the change in the gradient, as opposed to the gradient itself. We first present the \texttt{Adaptively Quantized Gradient Descent} (\texttt{AQGD}) algorithm for optimization of general smooth and strongly-convex loss functions. For pedagogical ease, we begin with the simpler case when the gradients can be calculated by the agent exactly and show in Theorem~\ref{thm:AQGD} that for globally smooth and strongly convex objective functions, \texttt{AQGD} guarantees exponentially fast convergence to the optimal solution. Interestingly, and perhaps counter-intuitively, the convergence rate is not hurt above a minimal value for $\bar{B}$, in the sense that the exponent of convergence is \textit{exactly} the same as that of unquantized gradient descent. Moreover, the minimal value of $\bar{B}$ that we identify matches (up to a universal constant) with that identified in a converse result~\cite{kostina} as being necessary to achieve the same rate as that of unquantized gradient descent. To extend this result for the LQR problem, which is neither strongly-convex nor globally smooth, in Theorem~\ref{thm:PL}, we prove that Theorem~\ref{thm:AQGD} is also valid under the weaker assumption of gradient-domination. This part of the paper was presented in a preliminary form at the L4DC 2024 conference \cite{mitra2024towards}.

$\bullet$ {\bf Local Assumptions with Noisy Gradients.} We next consider the case when the gradient is noisy (i.e., the gradient calculation at the agent is based on the observed system trajectories), and the objective function only possesses a local smoothness property. An additional challenge for the LQR problem is that the control policy that is generated must remain within the set of stabilizing policies. Towards this end, in Section~\ref{sec:noisy AQDG}, we introduce a variant of \texttt{AQGD} (termed as \texttt{NAQGD}) for optimization of general functions that are locally smooth and gradient dominant and the algorithm only uses noisy gradients. We prove in Theorem~\ref{thm:noisy gradient} that \texttt{NAQGD} converges exponentially fast to a neighborhood of the optimal policy whose range is characterized by the noise level of the noisy gradient while generating a sequence of stabilizing policies. The \texttt{NAQGD} algorithm achieves the same convergence performance as unquantized gradient descent with noisy gradients (under the locally smooth and gradient-dominant assumptions), as studied in \cite{cassel2021online}. Our proofs in Sections~\ref{sec:AQGD}-\ref{sec:noisy AQDG} rely on the construction of novel Lyapunov functions that simultaneously account for the optimization error, the error introduced by quantization and the possibly noisy gradients in our algorithm. 

$\bullet$ {\bf Application to Model-Free LQR and Sample Complexity Results.} In Section~\ref{sec:PG model free LQR}, we provide a method to compute a noisy gradient of the LQR objective function for a given policy based solely on trajectories collected from the target system. Based on this method, we show in Theorem~\ref{thm:noisy gradient for LQR} that \texttt{NAQGD} can be applied to the model-free LQR setting with the same convergence performance as Theorem~\ref{thm:noisy gradient}. Further, we provide in Corollary~\ref{coro:sample complexity} a finite-sample analysis of \texttt{NAQGD} that characterizes the number of data samples (on the collected trajectories) from the target system required to achieve a certain convergence performance of the algorithm.

Given that we relate our problem to one of optimization over rate limited channels, we would also like to comment on the literature in that area. Despite the rich literature that has emerged on the topic of communication-constrained optimization~\cite{bernstein, reddystich, stichsparse, gandikota, mayekar, saha, EF1}, the only paper we are aware of that manages to preserve fast linear rates (despite quantization) is the recent work by~\cite{kostina}. Our work differs from that of~\cite{kostina} both algorithmically, and also in terms of results: unlike the algorithm proposed in~\cite{kostina}, our proposed \texttt{AQGD} algorithm does not require maintaining any auxiliary sequence. Instead, it relies on carefully exploiting smoothness of the loss function to encode innovation signals. In addition to being conceptually simpler, \texttt{AQGD} preserves rates under weaker assumptions of local smoothness and gradient-domination. Moreover, while we also consider scenarios with noisy gradients, it is unclear whether the algorithm in~\cite{kostina} is applicable beyond the setting where exact noiseless gradients are available.  As such, our proposed technique might be of independent interest to the literature on compressed optimization. For a summary of the state-of-the-art in this area, we refer the interested reader to~\cite{EF2}.

{\bf Notation.} Let $[n]=\{1,\dots,n\}$ for $n\in\bbZ_{\ge1}$. Let $\mathbb{S}_{++}^n$ denote the set of all positive definite $n\times n$ matrices. For a vector $x\in\R^n$, let $\Vert x\Vert$ be its Euclidean norm. For a matrix $P\in\R^{n\times m}$, let $\Vert P\Vert$ and $\Vert P\Vert_F$ be its spectral norm and Frobenius norm, respectively. Let $I_n$ be the $n$ by $n$ identity matrix. Let $\mathcal{B}_d(0,R)$ denote the $d$-dimensional Euclidean ball of radius $R$ centered at the origin.

\section{Problem Formulation and Preliminaries}
\label{sec:LQR setup}
Consider a linear time-invariant (LTI) system given by 
\begin{equation}
x_{k+1} = Ax_k+Bu_k+w_k,
\label{eqn:LTI}
\end{equation}
where $A\in\mathbb{R}^{n\times n}$ and $B\in\mathbb{R}^{n\times m}$ are system matrices, $x_k\in\mathbb{R}^{n}$ is the state, $u_k\in\mathbb{R}^{m}$ is the control input and $w_0,w_1,\dots$ are i.i.d. Gaussian disturbances with zero mean and covariance $\Sigma_w\in\mathbb{S}_{++}^n$, i.e., $w_k\overset{i.i.d.}{\sim}\calN(0,\Sigma_w)$. We assume without loss of generality that $x_0=0$. The goal of the LQR problem is to compute the control inputs $u_0,u_1,\dots,$ that solve 
\begin{equation}\label{eqn:LQR obj}
\min_{u_0,u_1,\dots}\lim_{N\to\infty}\frac{1}{N}\E\Big[\sum_{k=0}^{N-1}x_k^{\top}Qx_k+u_k^{\top}Ru_k\Big],
\end{equation}
where $Q\in\mathbb{S}_{++}^n$ and $R\in\mathbb{S}_{++}^m$ are cost matrices, and the expectation is taken with respect to the disturbance process $\{w_k\}_{k\ge0}$. A classic result in optimal control yields that the optimal solution to the LQR problem in \eqref{eqn:LQR obj} is given by the static state-feedback control policy $u_t=Kx_t$ for some controller $K\in\mathbb{R}^{m\times n}$ (see, e.g., \cite{bertsekas2015dynamic}). Hence, solving \eqref{eqn:LQR obj} is equivalent to solving the following: 
\begin{equation}\label{eqn:LQR obj J(K)}
\min_{K\in\mathbb{R}^{m\times n}}J(K)\triangleq\lim_{N\to\infty}\frac{1}{N}\E\Big[\sum_{k=0}^{N-1}x_k^{\top}(Q+K^{\top}RK)x_k\Big].
\end{equation}
Additionally, another well-known result is that the cost function $J(K)$ is finite if and only if $K$ is stabilizing, i.e., the matrix $(A+BK)$ is Schur-stable \cite{bertsekas2015dynamic}. It follows that one can solve \eqref{eqn:LQR obj J(K)} by minimizing $J(K)$ over the set of stabilizing $K$. For any stabilizing $K$, we also know from \cite{bertsekas2015dynamic} the cost $J(K)$ yields the following closed-form expressions :
\begin{equation}\label{eqn:expression for J(K)}
J(K) = \textrm{trace}(P_K\Sigma_w)= \textrm{trace}\big((Q+K^{\top}RK)\Sigma_K\big),
\end{equation}
where $P_K\in\bbS_{++}^n$ is the solution to the following Riccati equations:
\begin{equation}
P_K = Q+K^{\top}RK+(A+BK)^{\top}P_K(A+BK), \hspace{1mm} 
\Sigma_K = \Sigma_w+(A+BK)\Sigma_K(A+BK)^{\top}.
\label{eqn:DARE}
\end{equation}

Based on the above arguments, one can use the so-called policy gradient method to obtain $K^*=\argmin_{K}J(K)$ (see, e.g., \cite{maartensson2009gradient,malik2020derivative,hu,bu2020topological,fatkhullin2021optimizing}). The basic scheme is to initialize with an arbitrary stabilizing $K_0$ and iteratively perform updates of the form: $K_{t+1}=K_{t}-\alpha \nabla J(K_t)$ for $t=0,1,\dots$, where $\alpha\in\mathbb{R}_{>0}$ is a suitably chosen step size. When the system matrices $A,B$ (and the cost matrices $Q,R$) are known, the exact gradient $\nabla J(K_t)$ of $J(K_t)$ can be computed and the algorithm described above converges exponentially fast to the optimal policy $K^*$, even though the cost $J(K)$ is not strictly convex \cite{fazel,malik2020derivative}. Further, even when the system model $A,B$ is unknown, for a given stabilizing $K$, one can compute a noisy version of $\nabla J(K)$, denoted as $\widehat{\nabla J(K)}$, based on observed system trajectories of \eqref{eqn:LTI} obtained by applying the control policy $u_t=Kx_t$; hence the name {\it model-free} learning for LQR \cite{fazel}. The policy gradient update now takes the form $K_{t+1}=K_t-\alpha\widehat{\nabla J(K_t)}$, whose convergence performance is characterized by the distance $\Vert \nabla J(K_t)-\widehat{\nabla J(K_t)}\Vert_F$ \cite{fazel,cassel2021online}.

\textbf{Objective.} We aim to study model-free learning for the LQR problem under communication constraints on the policy gradient updates (see the setup in Fig.~\ref{fig:Setup}). At each policy gradient update step $t$, the worker agent computes a potentially noisy policy gradient $\widehat{\nabla J(K_t)}$ by interacting with the environment and collecting system trajectories under the policy $K_t$. It then sends an encoded (i.e., quantized) version of $\widehat{\nabla J(K_t)}$ to the decision maker (or server) through a channel that supports noise-free transmission of a finite number of bits $\bar{B}$ per use of the channel. The server updates the policy to $K_{t+1}$ based on the decoded gradient $g_t$. Our goal is to propose a quantized policy gradient algorithm that specifies the encoding scheme at the worker and the policy update rule at the decision-maker, and characterize the convergence performance of the proposed algorithm to the optimal solution $K^*$. Moreover, we seek to analyze the convergence rate of the proposed algorithm and relate it to the capacity $\bar{B}$ of the channel. The major challenge here lies in the fact that the channel distorts the policy gradients $\widehat{\nabla J(K_t)}$, which makes it difficult to directly adapt the analysis in the literature on policy gradients for model-free LQR (e.g., \cite{fazel,malik2020derivative,cassel2021online}) to our setting. Thus, \emph{it is a priori unclear whether the sequence of policies generated using such distorted policy gradients converge to $K^*$ or even remain stabilizing.} 

{\bf Outline for the Rest of the Paper.} To solve the problem described above, we take several steps. In Section~\ref{sec:AQGD}, we first connect our problem to a general communication-constrained optimization problem setup (depicted by Fig.~\ref{fig:Model}) that has been studied in \cite{gandikota, mayekar, kostina}. For this setting, we propose the \texttt{AQGD} algorithm and show that it manages to preserve convergence rates despite quantization, provided the channel capacity exceeds a certain finite threshold and one has access to exact noiseless gradients. Building on these developments, in Section~\ref{sec:noisy AQDG}, we introduce a variant of \texttt{AQGD} (termed \texttt{NAQGD}) that can handle perturbation errors on the gradients - as needed to capture the effect of noisy policy gradients. Finally, in Section~\ref{sec:PG model free LQR}, we show how \texttt{NAQGD} addresses the model-free LQR problem under bit constraints, and provide a finite-sample analysis that relates the number of data samples used for obtaining the noisy gradients to the convergence performance of \texttt{NAQGD}. 

\section{Adaptively Quantized Gradient Descent (\texttt{AQGD}) for Communication-Constrained Optimization}\label{sec:AQGD}
\begin{figure}[t]
\centering 
\includegraphics[width=0.5\linewidth]{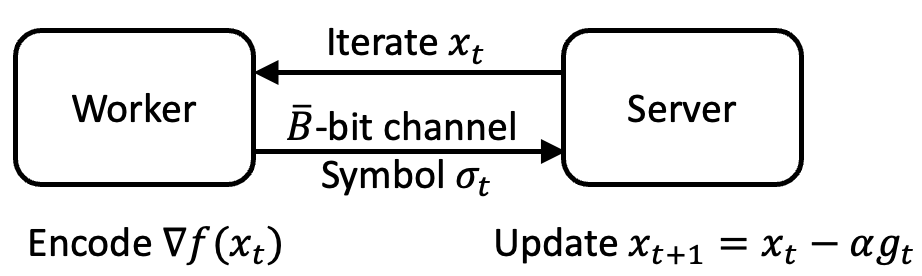}
\caption{Communication-constrained optimization setup with exact gradients at the worker, where $g_t$ represents the decoded gradient at the server.} 
\label{fig:Model}
\end{figure}

As we mentioned before, our problem is related to the general communication-constrained optimization framework in Fig.~\ref{fig:Model} that has been studied in various recent works~\cite{gandikota, mayekar, kostina, saha}. By assuming that the worker receives an {\it exact} gradient $\nabla f(x_t)$ of the objective function in each iteration $t$, we will first develop a method to tackle the gradient distortion (or error) introduced solely by the communication channel. The method developed in this section paves the way for our analysis later in Section~\ref{sec:noisy AQDG} that works for the case when the worker receives a noisy gradient $\widehat{\nabla f(x_t)}$, where we need to further consider the gradient error incurred due to noisy gradient evaluations. To proceed, we introduce the following standard definitions.

\begin{definition}\label{def:smooth}\textbf{(Smoothness)} A continuously differentiable function $f:\mathbb{R}^{d}\rightarrow \mathbb{R}$ is $L$-smooth if the gradient map $\nabla f:\mathbb{R}^{d}\rightarrow \mathbb{R}^d$ is $L$-Lipschitz, i.e.,
$$ \Vert \nabla f(x) - \nabla f(y) \Vert \leq L \Vert x-y \Vert, \forall x,y \in \mathbb{R}^d.$$
\end{definition}

\begin{definition} \textbf{(Strong-Convexity)} A differentiable function $f:\mathbb{R}^{d}\rightarrow \mathbb{R}$ is $\mu$-strongly convex if, given any $x,y\in\mathbb{R}^{d}$, the following inequality holds: 
$$
    f(y) - f(x) \geq\langle y-x, \nabla f(x) \rangle +\frac{\mu}{2}{\Vert y-x \Vert}^2.
$$
\end{definition}

Consider the setup in Fig.~\ref{fig:Model} with the objective $\min_{x\in\R^d}f(x)$, where $f:\mathbb{R}^d \to \mathbb{R}$. Under the assumptions that $f(\cdot)$ is globally $L$-smooth and $\mu$-strongly convex, it is well-known~\cite{nesterovlectures, bubeck2015convex} that when the channel from the worker to the server has infinite capacity, i.e., when $\bar{B}=\infty$, the vanilla gradient-descent algorithm $x_{t+1}=x_t-\alpha \nabla f(x_t)$ with the constant step-size $\alpha=1/L$ yields a so-called linear convergence guarantee for all $t\ge0$ given by
\begin{equation}
f(x_t)-f(x^*) \leq \left(1-\frac{1}{\kappa}\right)^t \left( f(x_0)-f(x^*) \right),
\label{eqn:basicGD}
\end{equation}
where $x^*=\argmin_{x\in\R^d}f(x)$ is the unique global minimizer and $\kappa=L/\mu$ is the condition number of $f(\cdot)$. In other words, $f(x_t)$ converges exponentially fast to the optimal solution at the rate $O(\exp(-t/\kappa))$.

When the channel only supports a finite number of bits $\bar{B}$, employing the update rule $x_{t+1}=x_t-\alpha g_t$ (where $g_t$ is a \textit{quantized} version of the gradient $\nabla f(x_t)$) may still guarantee convergence to $x^*$. However, the convergence rate in Eq.~\eqref{eqn:basicGD} may no longer be preserved. Indeed, the typical convergence rate achieved by almost all existing compressed/quantized gradient descent algorithms (see, e.g.,~\cite{reddystich,stichsparse}) is 
$$O\left(\exp\left(-\frac{t}{\kappa \delta}\right)\right),$$ where $\delta \geq 1$ captures the level of compression. A higher $\delta$ introduces more distortion, which causes the exponent of convergence to be a factor of $\delta$ slower than that of vanilla unquantized gradient descent. An exception is the recent work \cite{kostina} that employs a differentially quantized approach and preserves the convergence rate of the vanilla unquantized gradient descent. However, the algorithm proposed in \cite{kostina} requires maintaining an auxiliary sequence (besides $\{x_t\}_{t\ge0}$) and the convergence result relies on the (global) smoothness and strong convexity of the objective function. 

\begin{algorithm}[t]
\caption{Adaptively Quantized Gradient Descent (\texttt{AQGD})}
\label{algo:AQGD}  
\begin{algorithmic}[1] 
\State \textbf{Initialization:} $x_0=0$, $g_{-1}=0$, pick contraction factor $\gamma$ and pick $R_0$ such that $\Vert \nabla f(x_0) \Vert \leq R_0.$
\State \textbf{For} \hspace{0.25mm} {$t = 0, 1, \ldots, $} \hspace{0.25mm} \textbf{do}
\State \quad \textbf{At Worker:}
\State \quad Receive iterate $x_t$, gradient estimate $g_{t-1}$, and range $R_t$ from server. 
\State \quad Compute innovation $i_t= \nabla f(x_t)-g_{t-1}.$

\State \quad Encode the innovation $\tilde{i}_t = \mathcal{Q}_{b,R_t}(i_t).$ 
\quad 
\State \quad \textbf{At Decision-Maker/Server:}
\State \quad Decode $\tilde{i}_t,$ and estimate current gradient: $g_t=g_{t-1}+\tilde{i}_t.$
\State \quad Update the model as follows:
\begin{equation}
x_{t+1}=x_t-\alpha g_t.
\label{eqn:AQGD}
\end{equation}
\State \quad Update the range of the quantizer map as follows: 
\begin{equation}
R_{t+1}= \gamma R_{t}+\alpha L \Vert g_t \Vert.
\label{eqn:Range_Update}
\end{equation}
\State \textbf{End For}
\end{algorithmic}
 \end{algorithm}

\subsection{The \texttt{AQGD} Algorithm}\label{sec:algorithm design for AQGD}
In this section, we will introduce the Adaptively Quantized Gradient Descent (\texttt{AQGD}) algorithm; the detailed steps are included in Algorithm~\ref{algo:AQGD}. In contrast to the algorithm proposed in \cite{kostina}, the \texttt{AQGD} algorithm is conceptually simpler and does not require maintaining any auxiliary sequence. More importantly, we will show that \texttt{AQGD} preserves the convergence of unquantized gradient descent under weaker assumptions on the objective function (i.e., local smoothness and gradient domination). Additionally, we will show later in Section~\ref{sec:noisy AQDG} that \texttt{AQGD} can be further extended to handle noisy gradient calculations at the worker agent. 

Let us start by providing some high-level intuition behind the design of \texttt{AQGD}. The core idea is to quantize an innovation sequence corresponding to the gradient sequence $\{\nabla f(x_t)\}_{t\ge0}$, rather than directly quantizing $\{\nabla f(x_t)\}_{t\ge0}$. To see why this makes sense, let $g_{t-1}$ in Algorithm~\ref{algo:AQGD} represent the estimate of the gradient $\nabla f(x_{t-1})$ at the decision maker (or server) in iteration $t-1$. Supposing the function $f(\cdot)$ is smooth, the new gradient $\nabla f(x_t)$ at the worker cannot change abruptly from its previous value $\nabla f(x_{t-1})$ (as per Definition~\ref{def:smooth}). Based on this observation, if the decision-maker has a reasonably good estimate $g_{t-1}$ of the true gradient $\nabla f(x_{t-1})$ at iteration $t-1$ (i.e., $g_{t-1}$ is close to $\nabla f(x_{t-1})$), then $g_{t-1}$ should also be close to $\nabla f(x_t)$. Therefore, we define an innovation sequence $i_t=\nabla f(x_t)-g_{t-1}$, which captures the new information contained in the gradient $\nabla f(x_t)$ obtained by the worker at iteration $t$, relative to the most recent gradient estimate $g_{t-1}$ held by the decision-maker from iteration $t-1$. The worker encodes (i.e., quantizes) the innovation sequence $\{i_t\}_{t\ge0}$ based on the intuition that the terms in this sequence will eventually shrink in magnitude. We now discuss in detail the key components of the \texttt{AQGD} algorithm. 

$\bullet$ {\bf Quantization Scheme.} Line~6 of Algorithm~\ref{algo:AQGD} can be applied in tandem with any reasonably designed quantizer. A particularly simple instance of such a scheme is the scalar quantizer that we describe next. Suppose we want to quantize a vector $X \in \mathbb{R}^d$ with $\Vert X \Vert \leq R$. To achieve this, we can encode each component of $X$ separately using scalar quantizers. Since $\Vert X\Vert\le R$, we have $X_i \in [-R, R], \forall i \in [d]$, where $X_i$ is the $i$-th component of $X$. To encode each $X_i$, the scalar quantizer with $b$ bits partitions the interval $[-R,  R]$ into $2^b$ bins of equal width, and sets the center point $\tilde{X}_i$ of the bin containing $X_i$ to be the quantized version of $X_i$. This gives $\tilde{X}=[\tilde{X}_1, \ldots, \tilde{X}_d]^T$ as the quantized version of $X$. From the above arguments, we know that the scalar quantizer depends on both the range of each component $R$ and the number of bits $b$. Thus, we will use a quantizer map $\mathcal{Q}_{b,R}:\mathbb{R}^d \rightarrow \mathbb{R}^d$ parameterized by $R$ and $b$ to succinctly describe this scalar quantizer. Specifically, given $\Vert X \Vert \leq R$, we have $\tilde{X}=\mathcal{Q}_{b,R}(X)$. In the sequel, unless otherwise specified, we use the scalar quantizer described above in Algorithm~\ref{algo:AQGD}.  Later in Section~\ref{subsec:minimal}, we will argue that more complicated quantizers can be used to reduce the channel capacity required to achieve a certain level of convergence performance of the \texttt{AQGD} algorithm.

$\bullet$ {\bf Adaptive Ranges of the Quantizer.} Following our arguments above, Algorithm~\ref{algo:AQGD} recursively updates (i.e., decreases) the range $R_t$ of the quantizer map $\mathcal{Q}_{b,R_t}(\cdot)$. The intuition behind this step is as follows. Supposing our algorithm operates correctly, i.e., $x_t\to x^*$, we must have that $\nabla f(x_t) \rightarrow 0$. This, in turn, would imply that the gap between consecutive gradients in the sequence $\{\nabla f(x_t)\}_{t\ge0}$  should gradually shrink to $0$. Recalling that Algorithm~\ref{algo:AQGD} encodes the innovations in the sequence $\{\nabla f(x_t)\}_{t\ge0}$, this means that the innovation sequence $\{i_t\}_{t\ge0}$ should be contained in balls of progressively smaller radii. The key idea here is to maintain estimates of the radii of such balls, which in turn refine the quantizer range used to encode the innovation (as per line~10 of Algorithm~\ref{algo:AQGD}). Note that although the quantizer $\mathcal{Q}_{b,R_t}(\cdot)$ uses the same number of bits $b$ to encode each component of $i_t$ across the iterations in Algorithm~\ref{algo:AQGD}, a progressively finer quantization accuracy is afforded by the fact that the range $R_t$ shrinks with time.  We will leverage this property of the quantizer $Q_{b,R_t}(\cdot)$ in our analysis later. Note also that since we initialize $g_{-1}=0$, we have $i_0=\nabla f(x_0)$ and thus an initial upper bound $R_0$ on $\Vert\nabla f(x_0)\Vert$ needs to be known.

$\bullet$ \textbf{Correct Decoding at the Server.} To successfully run Algorithm~\ref{algo:AQGD}, we assume that the server knows the exact encoding strategy at the worker (i.e., knows the parameters $b$ and $R_t$). Then, given the $\bar{B}$-bit symbolic encoding of $\tilde{i}_t$, where $\bar{B}=bd$, the server can decode $\tilde{i}_t$ exactly.

\subsection{Convergence Analysis and Results for \texttt{AQGD}}
\label{sec:analysis}
Our first main convergence result pertains to the performance of \texttt{AQGD} under global assumptions of strong convexity and smoothness on the underlying function $f$. 

\begin{theorem} \label{thm:AQGD} (\textbf{Convergence of \texttt{AQGD}}) Suppose $f:\mathbb{R}^d \rightarrow \mathbb{R}$ is $L$-smooth and $\mu$-strongly convex. Suppose \texttt{AQGD} (Algorithm~\ref{algo:AQGD}) is run with step-size $\alpha=1/(6L)$ and contraction factor $\gamma=\sqrt{d}/2^b.$ There exists a universal constant $C \geq 1$ such that if the bit-precision $b$ per component satisfies
\begin{equation}
b \geq C \log \left( \frac{d\kappa}{\kappa-1} \right), 
\label{eqn:rate_req}
\end{equation}
then the following is true $\forall t \geq 0$: 
\begin{equation}
f(x_t)-f(x^*) \leq \left(1-\frac{1}{12\kappa}\right)^t \left( f(x_0)-f(x^*) + \alpha R^2_0\right),\hspace{2mm} \textrm{where} \hspace{2mm} \kappa = L/\mu.
\label{eqn:AQGD_bnd}
\end{equation}
\end{theorem}

\textbf{Discussion.} Comparing Eq.~\eqref{eqn:AQGD_bnd} to Eq.~\eqref{eqn:basicGD} reveals that \texttt{AQGD} \textit{preserves the exact same linear rate of convergence (up to universal constants) as vanilla unquantized gradient descent}, provided the channel capacity $\bar{B}=bd$ satisfies the requirement on $b$ in Eq.~\eqref{eqn:rate_req}. As mentioned earlier, commonly used compression schemes (including sophisticated ones like error-feedback~\cite{stichsparse}) scale down the exponent of the convergence rate by a factor $\delta \geq 1$ that captures the level of compression. The main takeaway then from Theorem~\ref{thm:AQGD} is that our simple scheme based on adaptive quantization can avoid such a scale-down, \textit{without the need for maintaining an auxiliary sequence as in the algorithm proposed in~\cite{kostina}}. 

The authors in~\cite{kostina} prove a converse result showing that to match the rate of unquantized gradient descent, a necessary condition on the bit-rate is 
\begin{equation}
 b \geq \log \left( \frac{\kappa+1}{\kappa-1} \right).
 \label{eqn:min_bit}
\end{equation}
We see from Eq.~\eqref{eqn:rate_req} that there is an extra $\log(d)$ factor compared to the \emph{minimal} rate in  Eq.~\eqref{eqn:min_bit}. This extra factor is due to the choice of a scalar quantizer in Algorithm~\ref{algo:AQGD} as we discussed above. By using a more complicated vector quantizer in \texttt{AQGD}, we will show in Section~\ref{subsec:minimal} that the extra $\log(d)$ factor in Eq.~\eqref{eqn:rate_req} can be shaved off.

As mentioned before, another merit of the \texttt{AQGD} algorithm is that its convergence analysis can be extended gracefully to functions that are locally smooth and gradient-dominant, and its algorithm design can be extended to handle noisy gradients. These developments are then applicable to the model-free LQR setting under communication constraints (see Section~\ref{sec:noisy AQDG}). As an initial step to build towards the LQR problem of eventual interest to us, we formally show in the next theorem that the convergence result in Theorem~\ref{thm:AQGD} can be generalized to (potentially non-convex) functions with the gradient-domination property. 

\begin{theorem} \label{thm:PL} Suppose $f:\mathbb{R}^d \rightarrow \mathbb{R}$ is $L$-smooth and satisfies the following gradient-domination property:
\begin{equation}
 \Vert \nabla f(x) \Vert^2 \geq 2\mu (f(x)-f(x^*)), \forall x \in \mathbb{R}^d,
 \label{eqn:PL}
\end{equation}
where $x^* \in \argmin_{x\in \mathbb{R}^d} f(x)$. Let $\alpha, \gamma$, and the bit-precision $b$ be chosen as in Theorem~\ref{thm:AQGD}. Then, \texttt{AQGD} provides exactly the same guarantee as in Eq.~\eqref{eqn:AQGD_bnd}. 
\end{theorem}

 Since any strongly convex function is also gradient dominant, it suffices to prove the stronger result in Theorem~\ref{thm:PL}, which naturally applies to the result in Theorem~\ref{thm:AQGD}. While the detailed proof of Theorem~\ref{thm:PL} can be found in Appendix~\ref{app:PLproof}, we provide a sketch of the main ideas below. 

\textbf{Proof Sketch for Theorem~\ref{thm:PL}}. The key technical novelty of this proof is a carefully designed Lyapunov (potential) function candidate, which we show to contract over time. The choice of the Lyapunov function candidate is given by
\begin{equation}
V_t \triangleq z_t + \alpha R^2_t, \hspace{2mm} \textrm{where} \hspace{2mm} z_t=f(x_t)-f(x^*). 
\label{eqn:Lyap}
\end{equation}
Note that when analyzing the convergence of vanilla unquantized gradient descent, it suffices to use $z_t$ (defined in Eq.~\eqref{eqn:Lyap}) as the Lyapunov function \cite{bubeck2015convex}. However, such a choice is not enough to prove our convergence result due to the following reasons: (i) the dynamics of the iterate $x_t$ are intimately coupled with the errors induced by quantization; and (ii) we also need to ensure that the errors due to quantization decay at a certain rate so that the algorithm achieves the desired overall convergence rate. In our proof, we show that the errors due to quantization are upper bounded by the dynamic range $R_t$ of the quantizer. These motivate the choice of the Lyapunov function candidate in Eq.~\eqref{eqn:Lyap}, which depends on the \emph{joint evolution} of $x_t$ and $R_t$. 

Next, we argue that $V_t$ decays to $0$ exponentially fast at the same rate as that of unquantized gradient descent, i.e., at the rate in Eq.~\eqref{eqn:AQGD_bnd}. To achieve this, we will rely on the following two intermediate lemmas.

\begin{lemma} \label{lemma:quant} (\textbf{No Overflow and Quantization Error})  Suppose $f$ is $L$-smooth. The following are then true for all $t\geq 0$: (i) $\Vert i_t \Vert \leq R_t$; and (ii) $\Vert e_t \Vert \leq \gamma R_t,$ where $e_t=\nabla f(x_t)-g_t$.
\end{lemma}

Part~(i) of Lemma~\ref{lemma:quant} shows that the innovation $i_t$ belongs to the ball $\mathcal{B}_d(0,R_t)$ for all $t\ge0$. Part~(ii) of Lemma~\ref{lemma:quant} tells us that the quantization error $e_t$ can be upper bounded by the dynamic range $R_t$ of the quantizer. It follows that if $R_t \rightarrow 0$, then $e_t \rightarrow 0$. 

\begin{lemma} (\textbf{Recursion for Dynamic Range}) \label{lemma:range} Suppose $f$ is $L$-smooth. If $\alpha$ is such that $\alpha L \leq 1$, then for all $t \geq 0$, we have:
\begin{equation}
R^2_{t+1} \leq 8 \gamma^2 R^2_t + 2 \alpha^2 L^2 \Vert \nabla f(x_t) \Vert^2. 
\end{equation}
\end{lemma}

Lemma~\ref{lemma:range} also reveals that the dynamic range $R_t$ depends on the magnitude of the gradient $\nabla f(x_t)$. Thus, to understand how the behavior of $R_t$ relates to that of $x_t$, we need the choice of the potential function $V_t$ in Eq.~\eqref{eqn:Lyap}. Now, using the two lemmas above and under the assumptions of smoothness and gradient domination, we show that 
$$ V_{t+1} \leq \underbrace{\left(1- \frac{\alpha \mu}{2} \right) z_t}_{T_1} + \underbrace{9 \alpha \gamma^2 R^2_t.}_{T_2}$$
In the above display, $T_1$ represents the optimization error that corresponds to the convergence of $f(x_t)$ to $f(x^*)$, and $T_2$ represents the quantization error. Hence, to achieve the final rate in Eq.~\eqref{eqn:AQGD_bnd}, we need the quantization error to decay faster than the optimization error. This can be achieved by choosing the bit-rate $b$ as in Eq.~\eqref{eqn:rate_req} so that the contraction factor $\gamma=\sqrt{d}/2^b$ is small enough. 

\subsection{Achieving Minimal Bit-Rates using $\epsilon$-net Coverings}
\label{subsec:minimal}
As mentioned earlier, we see from Eq.~\eqref{eqn:min_bit} and Eq.~\eqref{eqn:rate_req} that the channel capacity above which \texttt{AQGD} preserves the same bounds as unquantized gradient descent contains an extra factor $\log(d)$ compared to the minimal required capacity shown in \cite{kostina}. In this section, we will show that this extra $\log(d)$ factor can be shaved off by a more complicated quantization scheme (compared to the uniform scalar quantizer discussed in Section~\ref{sec:algorithm design for AQGD}). Specifically, we will show that circular across-the-dimensions quantizers are more efficient than rectangular dimension-by-dimension vector quantizers as the number of dimensions of the target vector increases. Formally, we introduce the following notion of an $\epsilon$-net~\cite{vershynin}. 

\begin{definition} Consider a subset $\mathcal{K} \subset \mathbb{R}^{d}$ and let $\epsilon > 0$. A subset $\mathcal{N} \subseteq \mathcal{K}$ is called an $\epsilon$-net of $\mathcal{K}$ if every point in $\mathcal{K}$ is within a distance of $\epsilon$ of some point of $\mathcal{N}$, i.e., 
$ \forall x \in \mathcal{K}, \exists x_0 \in \mathcal{N}: {\Vert x-x_0 \Vert} \leq \epsilon.$
\end{definition}

Based on the $\epsilon$-net  defined above, we modify line~6 of Algorithm~\ref{algo:AQGD} to construct a $\gamma R_t$-net of the ball $\mathcal{B}_d(0, R_t)$, where $\gamma \in (0,1)$ is again an input parameter to the algorithm that will be specified shortly. To be more specific, at each iteration $t$ of Algorithm~\ref{algo:AQGD}, the encoding region $\mathcal{B}_d(0, R_t)$ is covered by a union of balls each of radius $\gamma R_t$. The center $\tilde{i}_t$ of the ball containing the innovation $i_t$ is determined, and a symbol corresponding to this center is transmitted to the decision-maker. We then have the following result for the \texttt{AQGD} algorithm when the encoding scheme described above is used. 

\begin{theorem} \label{thm:min_bit} Suppose $f:\mathbb{R}^d \rightarrow \mathbb{R}$ is $L$-smooth and satisfies the gradient-domination property in Eq.~\eqref{eqn:PL}. Consider the version of \texttt{AQGD} described above with 
\begin{equation}
\gamma = \frac{1}{3}\left(1-\frac{1}{\kappa}\right).
\label{eqn:gamma}
\end{equation}
Let the channel bit-rate $\bar{B}$ satisfy:
\begin{equation}
 \bar{B} \geq d \log \left(7 \left( \frac{\kappa+1}{\kappa-1} \right)\right).
 \label{eqn:net_rate}
\end{equation}
Then, with $\alpha \leq 1/(6L)$, the rate of convergence of this variant of \texttt{AQGD} is the same as that in Eq.~\eqref{eqn:AQGD_bnd}. 
\end{theorem}

\begin{proof}
The convergence proof is exactly the same as that for Theorem~\ref{thm:PL} in Appendix~\ref{app:PLproof}. Specifically, we first note that if $i_t \in  \mathcal{B}_d(0, R_t),$ then by definition of a $\gamma R_t$-net of $\mathcal{B}_d(0, R_t)$, it holds that $\Vert i_t - \tilde{i}_t \Vert \leq \gamma R_t$. In addition, we need the following condition to hold:
$$ 9 \gamma^2 \leq \left(1-\frac{\alpha \mu}{2}\right) = \left(1-\frac{1}{12 \kappa}\right).$$
One can check that the choice of $\gamma$ in Theorem~\ref{thm:min_bit} satisfies the above condition. 

Next, we need to argue that the bit-rate in Eq.~\eqref{eqn:net_rate} suffices to construct a $\gamma R_t$-net of $\mathcal{B}_d(0, R_t)$ at each iteration. To that end, we will require the concept of covering numbers.

\begin{definition} (\textbf{Covering Numbers}) Let $\mathcal{K}$ be a subset of $\mathbb{R}^d$. The smallest cardinality of an $\epsilon$-net of $\mathcal{K}$ is called the covering number of $\mathcal{K}$ and is denoted by $\mathcal{N}(\mathcal{K},\epsilon)$. Equivalently, $\mathcal{N}(\mathcal{K},\epsilon)$ is the smallest number of closed balls with centers in $\mathcal{K}$ and radii $\epsilon$ whose union covers $\mathcal{K}$. 
\end{definition}

We will rely on the following key result that relates the covering number of a set $\mathcal{K}$ to its volume. 

\begin{lemma} \label{lemma:covers}  (\textbf{Covering Numbers and Volume}) \cite[Proposition 4.2.12]{vershynin} Let $\mathcal{K}$ be a subset of $\mathbb{R}^d$, and $\epsilon >0$. Then,
$$ \mathcal{N}(\mathcal{K},\epsilon) \leq \frac{\vert \left(\mathcal{K}\oplus(\epsilon/2)\mathcal{B}_d(0,1)\right)\vert}{\vert (\epsilon/2)\mathcal{B}_d(0,1)\vert}. $$
Here, we used $|\mathcal{K}|$ to represent the volume of a set $\mathcal{K}$, and $\mathcal{K}_1 \oplus \mathcal{K}_2$ to denote the Minkowski sum of two sets $\mathcal{K}_1, \mathcal{K}_2 \subset \mathbb{R}^d$.  
\end{lemma}
 
Recall that our encoding strategy involves constructing a $\gamma R_t$ - net of the ball $\mathcal{B}_d(0,R_t)$. Invoking Lemma \ref{lemma:covers} with $\mathcal{K}=\mathcal{B}_d(0,R_t)$, we obtain:
\begin{equation}
    \begin{aligned}
     \mathcal{N}(\mathcal{B}_d(0,R_t),\gamma\-R_t) &\leq \frac{\vert \left(\mathcal{B}_d(0,R_t)\oplus(\gamma R_t/2)\mathcal{B}_d(0,1)\right)\vert}{\vert (\gamma R_t/2)\mathcal{B}_d(0,1)\vert}\\
     &= \frac{\vert \left(1+\gamma/2\right)R_t\mathcal{B}_d(0,1) \vert}{\vert (\gamma R_t/2)\mathcal{B}_d(0,1)\vert}\\
     & = \frac{\left[{\left(1+\gamma/2\right)R_t}\right]^d}{\left[\left(\gamma/2\right)R_t\right]^d}\\
     & = \left(\frac{2}{\gamma}+1\right)^d.
    \end{aligned}
\end{equation}
Plugging in the choice of $\gamma$ from Eq.~\eqref{eqn:gamma} into the above display and simplifying, we obtain:
$$
\mathcal{N}(\mathcal{B}_d(0,R_t),\gamma\-R_t) \leq \left(7 \left(\frac{\kappa +1}{\kappa -1}\right)\right)^d. 
$$
It follows that the maximum number of balls needed to cover the encoding region $\mathcal{B}_d(0, R_t)$ in any iteration $t$ is upper bounded by the right-hand side of the above display; the same bound also applies to the size of the binary alphabet $\Sigma_t$ required for encoding the innovation $\tilde{i}_t$ at any $t$. To realize such an alphabet $\Sigma_t$, the maximum number of bits we need is given by
$$ \log(|\Sigma_t|) \leq d \log \left(7 \left( \frac{\kappa+1}{\kappa-1} \right)\right),                     $$
where $|\Sigma_t|$ denotes the cardinality of the alphabet $\Sigma_t.$ This completes the proof. 
\end{proof}

Some remarks about Theorem~\ref{thm:min_bit} are now in order. First, we note that the channel capacity in Eq.~\eqref{eqn:net_rate} matches the minimal rate identified in \cite{kostina}. Second, given $R >0$, consider a quantization map $\mathcal{H}_{\gamma, R}:\mathbb{R}^d \rightarrow \mathbb{R}^d$ that satisfies the following bound for some $\gamma \in (0,1)$:
$$ \Vert \mathcal{H}_{\gamma, R}(x)-x \Vert \leq \gamma R, \forall x\in \mathbb{R}^d \hspace{2mm} \textrm{such that} \hspace{2mm} \Vert x \Vert \leq R.$$
Inspecting the convergence proof in Appendix~\ref{app:PLproof}, one can show that \emph{any} quantization map $\mathcal{H}_{\gamma,R}(\cdot)$ that satisfies the above contraction property can be used in tandem with \texttt{AQGD}. However, there is a natural trade-off between the complexity of implementing a particular quantization map and the best contraction factor one can hope to achieve with it. For instance, while the scalar quantizer is easier to implement compared to the $\epsilon$-net-based covering scheme described in this section, the latter leads to tighter contraction factors, which then translate to minimal bit-rates.\footnote{For a discussion on constructing such coverings, see \cite{dumer,verger} and the references therein.} Such a trade-off reveals that the \texttt{AQGD} algorithm proposed in this work allows considerable flexibility in terms of the choice of the quantization map.

\section{\texttt{AQGD} under Local Assumptions and with Noisy Gradients}\label{sec:noisy AQDG}
When the system matrices $A$ and $B$ are unknown, directly computing $\nabla J(K)$ for a given $K\in\R^{m\times n}$ is not possible and one may instead estimate $\nabla J(K)$ using the resulting trajectories of system~\eqref{eqn:LTI} when the control policy $K$ is applied. Such an estimate of $\nabla J(\cdot)$ can be viewed as a noisy version of $\nabla J(\cdot)$ (as we will argue in detail in Section~\ref{sec:PG model free LQR}). Thus, in this section, we first adapt our general design and analysis of Algorithm~\ref{algo:AQGD} to the setting when evaluating the gradient $\nabla f(\cdot)$ of the function $f:\R^d\to\R$ incurs noise, and then specialize to the model-free LQR setting in Section~\ref{sec:PG model free LQR}. 

To proceed, let $\widehat{\nabla f(x)}$ denote a noisy gradient of $f(\cdot)$ and we assume the error of $\widehat{\nabla f(x)}$ satisfies
\begin{equation}\label{eqn:upper bound on gradient noise}
\Vert\widehat{\nabla f(x_t)}-\nabla f(x_t)\Vert\le\varepsilon_t,\forall t=0,1,\dots,
\end{equation}
where $x_t$ is chosen in Algorithm~\ref{algo:AQGD} and $\varepsilon_t\in\R_{\ge0}$. As we argued above, Eq.~\eqref{eqn:upper bound on gradient noise} is a general noisy gradient model. We will show next in Section~\ref{sec:PG model free LQR} that when applying Algorithm~\ref{algo:AQGD} to the model-free LQR problem, one can find a noisy gradient $\widehat{\nabla J(\cdot)}$ of the objective function $J(\cdot)$ that fits into the model of Eq.~\eqref{eqn:upper bound on gradient noise}. In particular, the noisy gradient $\widehat{\nabla J(\cdot)}$ given in Section~\ref{sec:PG model free LQR} will be shown to satisfy the condition in Eq.~\eqref{eqn:upper bound on gradient noise} only {\it with high probability}. Nevertheless, our analysis developed in this section can be adapted to account for the high probability upper bound on the error of the noisy gradient $\widehat{\nabla J(\cdot)}$ in the LQR setting. Note also that Eq.~\eqref{eqn:upper bound on gradient noise} allows a time-varying bound $\varepsilon_t$ on the gradient noise, which readily captures the case of a uniform bound (i.e., $\varepsilon_t=\varepsilon$ for all $t\ge0$ and some $\varepsilon\in\bbR_{\ge0}$).

 To adapt Algorithm~\ref{algo:AQGD} to the noisy gradient setting described above, we make the following modifications to the algorithm:
 
$\bullet$ The computation of the innovation $i_t$ in line~5 of Algorithm~\ref{algo:AQGD} is changed into 
\begin{equation}\label{eqn:innovation noisy case}
i_t=\widehat{\nabla f(x_t)}-g_{t-1}.
\end{equation}

$\bullet$ The update rule of the range of the quantizer map in line~10 of Algorithm~\ref{algo:AQGD} is changed into 
\begin{equation}\label{eqn:Range_Update modified}
R_{t+1}=\gamma R_t+\alpha L\Vert g_t\Vert+\varepsilon_{t}+\varepsilon_{t+1},
\end{equation}
where the initial $R_0$ is picked such that $\widehat{\nabla f(x_0)}\le R_0$.

Even though setting $R_{t+1}$ as Eq.~\eqref{eqn:Range_Update modified} requires the knowledge of $\varepsilon_{t+1}$ and $\varepsilon_t$, we will show in the next section that when applying Algorithm~\ref{algo:AQGD} to the model-free LQR setting, one can compute $\varepsilon_t$ based on some prior knowledge of system~\eqref{eqn:LTI} (despite that the actual system matrices $A,B$ are  unknown). In the sequel, we will refer to the modified Algorithm~\ref{algo:AQGD} based on Eqs.~\eqref{eqn:innovation noisy case}-\eqref{eqn:Range_Update modified} as Noisy \texttt{AQGD} (\texttt{NAQGD}). The intuitions behind the above modifications are as follows. First, since the true gradient $f(x_t)$ is not available, the innovation $i_t$ is computed based on the noisy gradient $\widehat{\nabla f(x_t)}$ (rather than $\nabla f(x_t)$), where  $g_{t-1}$ is the gradient estimate at the beginning of iteration $t$ in Algorithm~\ref{algo:AQGD}. Next, since $g_t$ contains an extra error term introduced by the noisy gradient $\widehat{\nabla f(x_t)}$, the range $R_t$ of the quantizer needs to be increased to accommodate this extra error term.

We move on to characterize the convergence of \texttt{NAQGD} described above. Compared to our analysis in Sections~\ref{sec:analysis}, we need to resolve several extra challenges here. First, the noisy gradient $\widehat{\nabla f(x_t)}$ introduces an extra term in the error of the gradient estimate $g_t$ in \texttt{NAQGD}, while the error of the gradient estimate $g_t$ only depends on the quantization step in \texttt{AQGD}. Second, the LQR objective function $J(\cdot)$ defined in~\eqref{eqn:LQR obj J(K)} does not possess the global smoothness and gradient-domination properties as required by Theorem~\ref{thm:PL} (see, e.g., \cite{fazel}). Thus, we also need to extend our analysis in Section~\ref{sec:analysis} to objective functions with "nice" \emph{local} properties which are then applicable to the LQR problem. Finally, as we argued in Section~\ref{sec:LQR setup}, when finding $K^*\in\argmin_{K}J(K)$ using an iterative algorithm (e.g., Algorithm~\ref{algo:AQGD}), 
we need to ensure that all the iterates $K_0,K_1,\dots$ stay in the set of stabilizing controllers. Since $K$ is stabilizing if and only if $J(K)$ is finite \cite{bertsekas2015dynamic}, we impose a sublevel set of $f(\cdot)$ as a feasible set in the general problem. We introduce the following definition.
\begin{definition}\label{def:local smooth}(\textbf{Locally Smooth}) A function $f:\mathbb{R}^d\to\mathbb{R}$ is said to be locally $(L,D)$-smooth over $\mathcal{X}\subseteq\mathbb{R}^d$ if $\|\nabla f(x)-\nabla f(y)\| \le L\|x-y\|$ for all $x\in\mathcal{X}$ and all $y\in\mathbb{R}^d$ with $\|y-x\| \le D$.
\end{definition}

We now present our main result of this section; the detailed proof is included in Appendix~\ref{app:proofs for NAQGD}.

\begin{theorem}\label{thm:noisy gradient} (\textbf{Convergence of \texttt{NAQGD} under Local Properties}) Consider $f:\bbR^d\to\bbR_{\ge0}$ and $\calX=\{x\in\bbR^d:f(x)\le v\}$, where $v\in\bbR_{\ge0}$. Suppose $f(\cdot)$ is locally $(L,D)$-smooth over $\calX$, $\Vert\nabla f(x)\Vert\le G$ for all $x\in\calX$, and $f(\cdot)$ satisfies the following local gradient-domination property:
\begin{equation}
 \Vert \nabla f(x) \Vert^2 \geq 2\mu (f(x)-f(x^*)), \forall x \in\mathcal{X},
\end{equation}
where $x^* \in \argmin_{x\in \mathcal{X}} f(x)$. Suppose \texttt{NAQGD} is initialized with $x_0\in\mathbb{R}^d$ such that $f(x_0)\le v/4$, and run with step-size $\alpha\le\min\{D/(7G),1/(6L)\}$. Moreover, suppose the gradient noise $\varepsilon_t$ in Eq.~\eqref{eqn:upper bound on gradient noise} satisfies that for any $t\ge0$, $\varepsilon_t\le G$ and $\bar{\varepsilon}_t^2\le v\mu/60$, where
\begin{equation}\label{eqn:def of bar epsilon_t}
\bar{\varepsilon}_t^2\triangleq\max_{s\in\{0,\dots,t\}}\left\{\left(1-\frac{\alpha\mu}{3}\right)^{t-s}(8\varepsilon_{s-1}^2+6\varepsilon_{s}^2)\right\},
\end{equation}
with $\varepsilon_{s}\triangleq0$ if $s<0$. Let the contraction factor $\gamma$ be the same as in Theorem~\ref{thm:PL}. Then, for all $t\ge0$, $f(x_t)\le v/2$ and the following is true: 
\begin{equation}\label{eqn:convergence of NAQGD}
f(x_{t})-f(x^*)\le\max\left\{\left(1-\frac{\alpha\mu}{3}\right)^{t}(f(x_0)-f(x^*)),\frac{15\bar{\varepsilon}_{t-1}^2}{\mu}\right\}.
\end{equation}
\end{theorem}

\textbf{Discussion.} While the results in \cite{kostina} and Section~\ref{sec:analysis} hold for objective functions that enjoy global smoothness and strong convexity (or gradient domination) properties, the result in Theorem~\ref{thm:noisy gradient} holds more generally for functions that exhibit such properties only locally on the feasible set $\calX=\{x\in\R^d:f(x)\le v\}$.\footnote{Note that assuming $\varepsilon_t=0$ for all $t\ge0$, the \texttt{NAQGD} algorithm reduces to \texttt{AQGD} introduced in Section~\ref{sec:AQGD}.} More importantly, Theorem~\ref{thm:noisy gradient} holds for scenarios with noisy gradients that satisfy the bound in  Eq.~\eqref{eqn:upper bound on gradient noise}. The standard gradient-descent method (without any quantization) has also been studied under the noisy gradient scenario in, e.g., \cite{cassel2021online}, where $f(x_{t+1})=f(x_t)-\alpha\widehat{\nabla f(x_t)}$. Interestingly, despite the quantization step,  our algorithm (i.e., \texttt{NAQGD}) achieves the same convergence performance - as given by Eq.~\eqref{eqn:convergence of NAQGD} (up to universal constants) - as that in \cite{cassel2021online}. Elaborating more on Eq.~\eqref{eqn:convergence of NAQGD}, the overall convergence result is characterized by a term with linear convergence and a term that depends on the gradient noise $\varepsilon_t$. Thus,  an overall linear convergence result is achievable provided that the gradient noise $\varepsilon_t$ decays exponentially with $t$. Alternatively, Eq.~\eqref{eqn:convergence of NAQGD} shows that $f(x_t)$ converges exponentially fast to a neighborhood of $f(x^*)$ whose range is characterized by $\bar{\varepsilon}_t$. Finally, it is worth pointing out the following immediate corollary of Theorem~\ref{thm:noisy gradient} when the algorithm takes exact gradients.
\begin{corollary}\label{coro:AQGD local}
Suppose $\varepsilon_t$ in Eq.~\eqref{eqn:upper bound on gradient noise} satisfies $\varepsilon_t=0$ for all $t\ge0$ and consider the same settings as those in Theorem~\ref{thm:noisy gradient}. Then, for all $t\ge0$, $f(x_t)\le v/2$ and the following is true: 
\begin{equation*}
f(x_{t})-f(x^*)\le\left(1-\frac{\alpha\mu}{3}\right)^{t}(f(x_0)-f(x^*)).
\end{equation*}
\end{corollary}
Corollary~\ref{coro:AQGD local} can be viewed as a generalization of Theorem~\ref{thm:AQGD} when \texttt{AQGD} is applied to objective functions with only local smoothness and gradient-domination properties, and our proposed approach still achieves linear convergence rates.

\textbf{Proof Sketch of Theorem~\ref{thm:noisy gradient}.} Similarly to Theorem~\ref{thm:PL}, the proof relies on a carefully designed potential function $V_t$. In particular, since noisy gradients $\widehat{f(x_t)}$ are used in \texttt{NAQGD}, the choice of $V_t$ also depends on the noise level $\varepsilon_t$ of $\widehat{f(x_t)}$. In addition, to tackle the extra challenges introduced by the local properties of the objective function $f(\cdot)$, when upper bounding the innovation $i_t$ and the gradient error $e_t=g_t-\nabla f(x_t)$, we need a more careful induction argument to prove the desired convergence result. 

\section{Application to the Model-Free LQR}
\label{sec:PG model free LQR}
We are now in place to apply the modified Algorithm~\ref{algo:AQGD} described above (i.e., \texttt{NAQGD}) to the model-free LQR case, i.e., when the agent (depicted in Fig.~\ref{fig:Setup}) does not have access to the system model $A,B$ and thus {\it cannot} evaluate the true gradient $\nabla J(K)$ at any controller $K\in\R^{m\times n}$. For notational simplicity in the sequel, let us introduce the following:
\begin{align}
\beta_0 I\preceq R\preceq \beta_1 I,\ \beta_0 I\preceq Q\preceq \beta_1 I,\ \Sigma_w\succeq \sigma^2_w I,\ \Vert B \Vert \le\psi,\ J(K^*)\le \frac{J}{4},\label{eqn:parameters in J(K)}
\end{align}
where $\beta_0,\beta_1,\sigma_w,J\in\mathbb{R}_{>0}$, $\psi\in\mathbb{R}_{\ge1}$, and $J(K^*)$ is the optimal cost to problem~\eqref{eqn:LQR obj J(K)}. Moreover, we assume without loss of generality that $\beta_1\le1$ (since one may always scale the cost matrices $Q,R$ by a positive real number). In addition, we construct the following sublevel set of the LQR objective function $J(\cdot)$: 
\begin{equation}\label{eqn:feasible set for J(K)}
\mathcal{K}=\{K\in\mathbb{R}^{m\times n}:J(K)\le J\},
\end{equation}
and impose $\mathcal{K}$ as the feasible set of $J(\cdot)$. We will use the following lemma in our analysis.
\begin{lemma}\label{lemma:properties of J(K)} The objective $J(\cdot)$ in problem~\eqref{eqn:LQR obj J(K)} and $\calK$ defined in Eq.~\eqref{eqn:feasible set for J(K)} satisfy: \\
\noindent(a) \cite[Lemma~1]{fazel}\&\cite[Lemma~40]{cassel2020logarithmic} For any $K\in\mathcal{K}$, the gradient of $J(K)$ satisfies $\nabla J(K)=2\big((R+B^{\top}P_KB)K+B^{\top}P_KA\big)\Sigma_K$, where $P_K,\Sigma_K$ are given in~\eqref{eqn:DARE} and $\Vert P_K\Vert\le 2\beta_1\zeta^4/(1-\eta)$.\\
\noindent(b) \cite[Lemma~41]{cassel2020logarithmic} For any $K\in\mathcal{K}$, it holds that $\Vert (A+BK)^k \Vert\le\zeta\eta^k$ for all $k\in\mathbb{Z}_{\ge0}$ and $\Vert K \Vert \le\zeta$, where $\zeta\triangleq\sqrt{J/(\beta_0\sigma_w^2)}$ satisfies $\zeta\ge1$ and $\eta\triangleq1-1/(2\zeta^2)$.\\
\noindent(c) \cite[Lemma~25]{fazel} For any $K\in\mathcal{K}$, it holds that $\Vert \nabla J(K)\Vert_F\le G= \frac{2J}{\beta_0\sigma_w^2}\sqrt{(\sigma_w^2+\psi^2J)J}$.\\
\noindent(d) \cite[Lemma~5]{cassel2021online} Let $D=1/(\psi\zeta^3)$. Then, $J(\cdot)$ is $(D,L)$-locally smooth with $L=112\sqrt{n}J\psi^2\zeta^8/\beta_0$, i.e., $\Vert \nabla J(K^{\prime})-\nabla J(K)\Vert_F\le L \Vert K^{\prime}-K\Vert_F$,  
for all $K\in\mathcal{K}$ and all $K^{\prime}\in\mathbb{R}^{m\times n}$ with $\Vert K^{\prime}-K\Vert_F \le D$; $J(\cdot)$ is $(D,\bar{G})$-locally Lipschitz with $\bar{G}=4\psi J\zeta^7/\beta^0$, i.e., $|J(K^{\prime})-J(K)|\le\bar{G}\Vert K^{\prime}-K\Vert_F$ for all $K\in\calK$ and all $K^{\prime}\in\R^{m\times n}$ with $\Vert K^{\prime}-K\Vert_F\le D$.  \\
\noindent(e) \cite[Lemma~11]{fazel} $J(\cdot)$ satisfies the gradient-domination property with $\mu=2J/\zeta^4$, i.e., $\Vert \nabla J(K)\Vert^2_F\ge 2\mu(J(K)-J(K^*))$ for all $K\in\calK$, where $K^*=\argmin_{K\in\calK}J(K)$.
\end{lemma}
Note that Lemma~\ref{lemma:properties of J(K)}(c)-(e) show that the LQR objective function $J(\cdot)$ satisfies the local properties required by Theorem~\ref{thm:noisy gradient} over the set $\calK$. Note also that the additional local Lipschitz property of $J(\cdot)$ will play a role when analyzing the noisy gradient of $J(\cdot)$ as we elaborate next.

Lemma~\ref{lemma:properties of J(K)}(a) provides a closed-form expression of $\nabla J(K)$ which however depends on the system matrices $A$ and $B$. For unknown $A$ and $B$, we introduce in Algorithm~\ref{algo:gradient estimate} a noisy gradient oracle of $J(K_t)$ based solely on the observed system trajectories of system~\eqref{eqn:LTI} when the control policy $K_t$ is applied; similar noisy gradient oracles for model-free LQR have been considered in, e.g., \cite{cassel2021online,fazel,li2022distributed}. Specifically, for any iteration $t=0,1,\dots,$ of \texttt{NAQGD}, the worker agent obtains $\widehat{\nabla J(K_t)}$ from Algorithm~\ref{algo:gradient estimate} by playing $u_{t,k}^i=(K_t+U_t^i)x_{t,k}^i$ and observing $x_{t,k}^i$ for  $k=0,1,\dots,N_t-1$ from $\ell_t$ trajectories of system~\eqref{eqn:LTI} with length $N_t$, where $x_{t,k+1}^i=Ax_{t,k}^i+Bu_{t,k}^i+w_{t,k}^i$ and $U_t^i$ is a random purturbation.\footnote{Note that we use $t$, $i$ and $k$ to index an iteration of \texttt{NAQGD}, a trajectory of system~\eqref{eqn:LTI} and a time step in the trajectory used in Algorithm~\ref{algo:gradient estimate}, respectively. We assume that $w_{t,k}^i\overset{i.i.d.}\sim\calN(0,\Sigma_w)$ $\forall i\in[\ell_t]$, $\forall k\in\{0,\dots,N_t-1\}$ and $\forall t\in\{0,1,\dots\}$.} The following remark discusses some important aspects regarding the implementation and analysis of \texttt{NAQGD} when applied to solve model-free LQR. 

\begin{remark}\label{remark:random smoothing}As we will show in our proof,
Algorithm~\ref{algo:gradient estimate} relies on the random perturbation by $U_t^i$ to ensure the bound on the noisy gradient $\widehat{\nabla J(K_t)}$ required by Eq.~\eqref{eqn:upper bound on gradient noise} and Theorem~\ref{thm:noisy gradient}. Such random perturbations are typical in gradient-based methods without access to the true gradients \cite{flaxman2005online,nesterov2017random}. Note that both $K$ and $\nabla J(K)$ are matrices in $\mathbb{R}^{m\times n}$, while Algorithm~\ref{algo:AQGD} is designed for vectors $x,\nabla f(x)\in\mathbb{R}^d$. Nonetheless, one can first vectorize $K$ and $\nabla J(K)$ to obtain vectors in $\mathbb{R}^{m \times n}$, and then apply Algorithm~\ref{algo:AQGD} to achieve the desired convergence performance. Finally, using the results in \cite[Section~3.3.3]{vershynin}, one can show that the random matrices $U_t^i\in\R^{m\times n}$ obtained as line~3 of Algorithm~\ref{algo:gradient estimate} (or the vectorized $U_t^i$ in $\R^{mn\times1}$) are equal in distribution to a random vector sampled uniformly from the $nm$-dimensional unit sphere (i.e., $\{u:\Vert u\Vert=1\}$). 
\end{remark}

\begin{algorithm}[t]
\caption{Compute $\widehat{\nabla J(K_t)}$ for iteration $t$}
\label{algo:gradient estimate}  
\begin{algorithmic}[1]
\Statex\textbf{Input:} controller $K_t$, number of trajectories $\ell_t$, trajectory length $N_t$, parameter $r$.
\For{$i=1,\dots,\ell_t$}
\State \textbf{Initialization:} $x_{t,0}^i=0$.
\State Sample $\widetilde{U}_t^i$ with i.i.d. entries from $\calN(0,1)$ and normalize $U_t^i=\widetilde{U}_t^i/\Vert\widetilde{U}_t^i\Vert_F$.
\State Play $u_{t,k}^i=(K_t+rU_t^i)x_{t,k}^i$ and observe $x_{t,k}^i$ for all $k=0,\dots,N_t-1$.
\State Compute $\widehat{J}_i=\frac{mn}{rN_t}\sum_{k=0}^{N_t-1}(x_{t,k}^{i\top}Qx_{t,k}^i+u_{t,k}^{i\top}Ru_{t,k}^i)U_t^i$.
\EndFor
\Statex {\bf Output:} $\widehat{\nabla J(K_t)}=\frac{1}{\ell_t}\sum_{i=1}^{\ell_t}\widehat{J}_i$.
\end{algorithmic}
\end{algorithm}

We now state our result for \texttt{NAQGD} when applied to solve model-free LQR; the detailed proof of Theorem~\ref{thm:noisy gradient for LQR} is included in Appendix~\ref{app:proofs for NAQGD}.
\begin{theorem}
\label{thm:noisy gradient for LQR}
Consider the same setting as Theorem~\ref{thm:noisy gradient} with the parameters $D,G,L,\mu$ of the LQR objective $J(\cdot)$ given by Lemma~\ref{lemma:properties of J(K)} and $\calK$ given by Eq.~\eqref{eqn:feasible set for J(K)}. Suppose \texttt{NAQGD} is applied to solve $\min_{K\in\calK}J(K)$ with $\widehat{\nabla J(K_t)}$ given by Algorithm~\ref{algo:gradient estimate} for $t=0,\dots,T-1$, and initialized with $J(K_0)\le J/4$. Let the step-size $\alpha$ and the contraction factor $\gamma$ be the same as Theorem~\ref{thm:noisy gradient}. For any $0<\delta<1$, let 
\begin{multline}\label{eqn:epsilon_t for LQR}
\varepsilon_t=Lr+\frac{mnJ}{r\sqrt{\ell_t}}\sqrt{8\log\frac{45T}{\delta}}+\frac{10mn\beta_1\zeta^4\trace(\Sigma_w)}{3TN_tr(1-\eta)}\log\frac{27TN_t}{\delta}\\
+\frac{10mn\beta_1\zeta^2\trace(\Sigma_w)}{r(1-\eta)^2}\log\frac{3N_t\ell_tT}{\delta}\Big(\frac{1+\zeta^2}{\sqrt{\ell_t}}\sqrt{2\log\frac{45T}{\delta}}+\frac{\zeta^4}{N_t(1-\eta)}\Big),
\end{multline}
where $r$ is chosen to satisfy $r\le\min\{J/(2\bar{G}),D\}$. Moreover, suppose $N_t$ and $\ell_t$ are chosen such that $\varepsilon_t\le G$ and $15\bar{\varepsilon}_t^2/\mu\le J/4$ for all $t\in\{0,\dots,T-1\}$, where $\bar{\varepsilon}_t$ is defined as Eq.~\eqref{eqn:def of bar epsilon_t} using $\varepsilon_t$ given above. Then, with probability at least $1-\delta$, $J(K_t)\le J/2$ and 
\begin{equation}\label{eqn:convergence of model-free LQR}
J(K_{t})-J(K^*)\le\max\left\{\left(1-\frac{\alpha\mu}{3}\right)^{t}(J(K_0)-J(K^*)),\frac{15\bar{\varepsilon}_{t-1}^2}{\mu}\right\},\forall t\in\{0,\dots,T-1\},
\end{equation}
where $K^*\in\arg\min_{K\in\calK}J(K)$.
\end{theorem}

\textbf{Discussion.} Compared to the generic result provided in Theorem~\ref{thm:noisy gradient}, the extra step here is to specify the choice of $\varepsilon_t$ such that $\Vert \widehat{\nabla J(K_t)}-\nabla J(K_t)\Vert_F\le\varepsilon_t$ for the noisy gradient  $\widehat{\nabla J(K_t)}$ returned by Algorithm~\ref{algo:gradient estimate}. Since $\widehat{\nabla J(K_t)}$ is stochastic due to the stochastic disturbance $w_k$ in system~\eqref{eqn:LTI}, we show that the upper bound $\Vert \widehat{\nabla J(K_t)}-\nabla J(K_t)\Vert_F\le\varepsilon_t$ holds with high probability under the choice of $\varepsilon_t$ in Eq.~\eqref{eqn:epsilon_t for LQR} and consequently, the overall convergence result provided in Theorem~\ref{thm:noisy gradient for LQR} holds with high probability. Additionally, we see from the results in Lemma~\ref{lemma:properties of J(K)} that $\varepsilon_t$ in Eq.~\eqref{eqn:epsilon_t for LQR} can be determined using the parameters of the LQR problem, including $\beta_0,\beta_1,\sigma_w,\psi,J$ in Eq.~\eqref{eqn:parameters in J(K)}. Meanwhile, note that Eq.~\eqref{eqn:epsilon_t for LQR} can be equivalently written as $\varepsilon_t=\widetilde{\calO}(r+1/(\sqrt{\ell_t}r)+1/(N_tr))$, where $\widetilde{\calO}(\cdot)$ hides polynomial factors in problem parameters of LQR and logarithmic factors in $T,N_t,\ell_t$. Thus, to satisfy the requirement on $\varepsilon_t$ in Theorem~\ref{thm:noisy gradient for LQR}, the lengths $N_t$ and $\ell_t$ need to be greater than some polynomials in the problem parameters of LQR. Furthermore, setting the input parameter $t$ for Algorithm~\ref{algo:gradient estimate} to satisfy $r=\calO(1/(\ell_t)^{1/4})$ and setting the trajectory length to be sufficiently long such that $N_t\ge\sqrt{\ell_t}$, one can show that $\varepsilon_t=\widetilde{\calO}(1/(\ell_t)^{1/4})$.

Similar to our arguments before, the overall convergence of \texttt{NAQGD} (when applied to model-free LQR) is given by a term with linear convergence and a term that depends quadratically on the gradient noise $\varepsilon_t$. Alternatively, the convergence given by Eq.~\eqref{eqn:convergence of model-free LQR} in Theorem~\ref{thm:noisy gradient for LQR} can be interpreted as follows. Suppose we want to achieve a $\tau$-convergence result, i.e., $J(K_t)-J(K^*)\le\tau$ for some $\tau\in\R_{>0}$.  Letting $N_t$ and $\ell_t$ to be the same for all $t\in\{0,\dots,T-1\}$, we have $\varepsilon_t=\varepsilon$ for some $\varepsilon$ and for all $t$, and $\bar{\varepsilon}_t=14\varepsilon_t$ for all $t$ by Eq.~\eqref{eqn:def of bar epsilon_t}, which implies via Eq.~\eqref{eqn:convergence of model-free LQR} that 
\begin{equation*}
J(K_t)-J(K^{\star})\le\max\Big\{\Big(1-\frac{\alpha\mu}{3}\Big)^t(J(K_0)-J(K^*)),\frac{210\varepsilon^2}{\mu}\Big\}.
\end{equation*}
Moreover, we know from our arguments above that $\varepsilon=\widetilde{\calO}(1/(\ell_t)^{1/4})$, i.e., $\varepsilon^2=\widetilde{\calO}(1/\sqrt{\ell_t})$. One can now deduce  that to achieve the desired $\tau$-convergence result, the number of iterations for \texttt{NAQGD} should satisfy $T=\calO(\log(1/\tau))$ and the number of trajectories used in Algorithm~\ref{algo:gradient estimate} should satisfy $\ell_t=\calO((1/\tau)^2(\log(1/\tau))^2)$ for all $t\in\{0,\dots,T-1\}$. The above arguments yield the following sample complexity result which is a typical performance measure of learning algorithms for LQR (with unknown system model) \cite{tsiamis,hu}.
\begin{corollary}\label{coro:sample complexity}
It requires $\calO((1/\tau)^2(\log(1/\tau))^3)$ number of trajectories of system~\eqref{eqn:LTI} for the \texttt{NAQGD} algorithm to achieve a $\tau$-convergence result when applied to the communication-constrained model-free LQR.  
\end{corollary}
The above sample complexity result also matches with the result in \cite{malik2020derivative} (up to logarithmic factors in $1/\tau$), where a version of LQR with discounted cost was studied in \cite{malik2020derivative} for the model-free setting without any quantization, and the convergence result in \cite{malik2020derivative} only holds with a constant probability of $3/4$.

{\bf Proof Sketch of Theorem~\ref{thm:noisy gradient for LQR}.} The main step in the proof of Theorem~\ref{thm:noisy gradient for LQR} is to upper bound $\Vert\widehat{\nabla J(K_t)}-\nabla J(K_t)\Vert_F$, where $\widehat{\nabla J(K_t)}$ is a random matrix due to the random disturbance in system~\eqref{eqn:LTI}. Since the noisy gradient $\widehat{\nabla J(K_t)}$ is not an unbiased estimate of $\nabla J(K_t)$, $\Vert\widehat{\nabla J(K_t)}-\nabla J(K_t)\Vert_F$ may not be upper bounded by directly applying standard concentration inequalities. Thus, we prove the result by introducing an unbiased estimate of $\nabla J(K_t)$, denoted as $\nabla J^r(K_t)$, and decomposing $\Vert\widehat{\nabla J(K_t)}-\nabla J(K_t)\Vert_F\le\Vert\widehat{\nabla J(K_t)}-\nabla J^r(K_t)\Vert_F+\Vert\nabla J^r(K_t)-\nabla J(K_t)\Vert_F$. We then provide upper bounds on the two resulting terms using concentration inequalities for martingales.

\section{Numerical Results}
To validate our theoretical results, we apply \texttt{AQGD} and \texttt{NAQGD} to solve instances of the LQR problem in \eqref{eqn:LQR obj J(K)} constructed below. Specifically, we generate a Schur-stable $A\in\bbR^{5\times 5}$ matrix in a random manner and a $B\in\bbR^{5\times 3}$ matrix is also generated randomly. The cost matrices are set to be $Q=5I_{5\times 5}$ and $R=5I_{3\times 3}$. The noise covariance is set to be $\Sigma_w=I_5$. We first apply \texttt{AQGD} (Algorithm~\ref{algo:AQGD}) to solve the above LQR instances when the exact gradient $\nabla J(K)$ is available for a given stabilizing $K^{3\times 5}$. Note that $K_0=0$ is an initial stabilizing controller since $A$ is stable. The parameters in \texttt{AQGD} are set according to Corollary~\ref{coro:AQGD local}, where the parameters of the LQR problem are given by Lemma~\ref{lemma:properties of J(K)}. In particular, we set the step size $\alpha=10^{-3}$ and set $R_0=G$ (since $\Vert J(K_0)\Vert_F\le G$ by Lemma~\ref{lemma:properties of J(K)}(c)). The performance of \texttt{AQGD} is presented in Fig.~\ref{fig:convergence}(a), which shows the exponential convergence of $J(K_t)$ to the optimal solution $J(K^{\star})$, aligning with the result in Corollary~\ref{coro:AQGD local} and also our findings in Sections~\ref{sec:AQGD}-\ref{sec:noisy AQDG}. The fluctuation of the curve in Fig.~\ref{fig:convergence} is due to the fact that when proving the convergence of \texttt{AQGD}, we only show that $V_t=J(K_t)-J(K^*)+\alpha R_t$ decays exponentially, i.e., $J(K_t)-J(K^*)$ can potentially increase in some iterations of \texttt{AQGD}. 

Next, we apply \texttt{NAQGD} (described in Section~\ref{sec:noisy AQDG}) to the LQR instances constructed above using the noisy gradient $\widehat{\nabla J(K_t)}$ returned by Algorithm~\ref{algo:gradient estimate}, where the parameters in \texttt{NAQGD} and Algorithm~\ref{algo:gradient estimate} are set according to Theorem~\ref{thm:noisy gradient for LQR}. In particular, we set $\alpha=0.2\times 10^{-3}$, $r=0.1$, $\ell_t=\ell\ge30(1/\tau)^2(\log(1/\tau))^3$ (with $\tau=0.1$) and $N_t\ge\ell_t$ for all $t\in\{0,\dots,40\}$, which implies that $\varepsilon_t=\varepsilon$ for all $t\in\{0,\dots,40\}$ as per Eq.~\eqref{eqn:epsilon_t for LQR}. We further set $R_0=\varepsilon+G$ such that $\Vert\widehat{J(K_0)}\Vert_F\le R_0$ (see our arguments in the proof of Theorem~\ref{thm:noisy gradient for LQR}). While Fig.~\ref{fig:convergence}(a) shows that \texttt{AQGD} converges to $J(K^*)$, Fig.~\ref{fig:convergence}(b) shows that \texttt{NAQGD} converges (exponentially) to only a neighborhood around $J(K^*)$, which aligns with the result in Theorem~\ref{thm:noisy gradient for LQR}. In addition, since the noisy gradients are used in \texttt{NAQGD}, we choose a more conservative step size $\alpha$ compared to that chosen in \texttt{AQGD}, which leads to a slower convergence rate in Fig.~\ref{fig:convergence}(b).

\begin{figure}[htbp]
\centering
\subfloat[a][Performance of \texttt{AQGD}]{\includegraphics[width=0.45\linewidth]{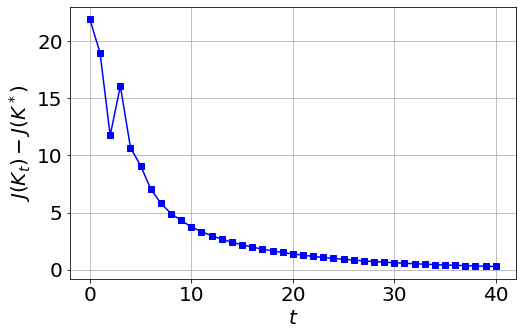}} 
\subfloat[b][Performance of \texttt{NAQGD}]{\includegraphics[width=0.45\linewidth]{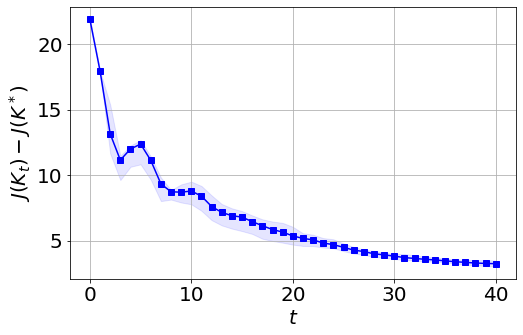}}
\caption{The suboptimality gap $J(K_t)-J(K^*)$ versus the iteration $t$ in the \texttt{AQGD} and \texttt{NAQGD} algorithms. In Fig.~\ref{fig:convergence}(b), the results are averaged over $10$ experiments and the shaded regions represent quantiles.}
\label{fig:convergence}
\end{figure}

\section{Conclusions and Future Directions}
We studied policy gradient algorithms for the model-free LQR problem subject to communication constraints. Specifically, we considered a rate-limited channel and introduced a novel adaptively quantized gradient-descent algorithm titled \texttt{AQGD}. We showed that under assumptions of smoothness and gradient-domination, \texttt{AQGD} guarantees exponentially fast convergence to the globally optimal solution. More importantly, above a finite bit-rate, the exponent of convergence of \texttt{AQGD} remains unaffected by quantization. We further introduced a variant of \texttt{AQGD} that works under noisy gradients and local assumptions, and applied it to solve the model-free LQR problem, providing convergence and finite-sample guarantees along the way. Overall, our work can be seen as an initial attempt towards merging model-free control with the area of networked control systems. There are several open questions left open by our work. We discuss some of them below. 

\begin{itemize}
\item \textbf{Adaptive Quantization for other optimization and RL problems.} One key message conveyed by our work is the power of adaptive quantization. Using very simple guiding principles, we could craft algorithms that build on this powerful idea and incur no loss in performance above certain minimal requirements on the channel. This naturally begs the question: Can the ideas developed in this paper be applied more broadly to other classes of optimization and RL problems subject to bit-constraints? Our conjecture is that for stochastic approximation problems where the underlying operator exhibits some version of Lipschitzness (e.g., like the gradients in smooth optimization), similar ideas should likely go through. 

\item \textbf{Alternate Channel Models.} The bit-constrained channel model we studied in this work is arguably one of the simplest channel models. One natural direction is to explore other well-known models in the networked controls literature: packet-dropping models, noisy models, and channels that introduce latencies. In general, the interplay between different types of communication channels and the finite-time performance of RL algorithms is not well understood. 

\item \textbf{Communication-Efficient Multi-Agent RL.} The rationale for looking at communication constraints in this paper was broadly motivated by the areas of federated/multi-agent RL and RL-based control over networks. In particular, while some recent works have considered multi-agent versions of the LQR problem~\cite{wang2023model, shinMARL}, the aspect of how communication affects performance is not adequately explored in such works. 
\end{itemize}
\bibliography{refs}
\bibliographystyle{unsrt}
\newpage
\appendix
\section{Proof of Theorem~\ref{thm:PL}}
\label{app:PLproof}
In this section, we will provide a proof for Theorem~\ref{thm:PL}, which relies on a few intermediate lemmas. The first one below provides a bound on the performance of the scalar quantizer introduced in Section~\ref{sec:algorithm design for AQGD}.

\begin{lemma} \label{lemma:SQ} (\textbf{Scalar Quantizer Bound}) Given a vector $X \in \mathbb{R}^d$ such that $\Vert X \Vert \leq R$, let  $\tilde{X}=\mathcal{Q}_{b,R}(X).$ The following is then true:
$$ \Vert \tilde{X} - X \Vert \leq \gamma R,$$
where $\gamma = \frac{\sqrt{d}}{2^b}.$ 
\end{lemma}
\begin{proof}
The proof is straightforward, and we only provide it here for completeness. Let $X_i$ and $\tilde{X}_i$ be the $i$-th components of $X$ and $\tilde{X}$, respectively. Since $\Vert X \Vert \leq R$, clearly $|X_i| \leq R, \forall i \in [d]$. From the description of the scalar quantizer in Section~\ref{sec:AQGD}, if $b$ bits are used to encode $X_i$, then it is easy to see that $|X_i - \tilde{X}_i| \leq 2^{-b} R.$ We thus have:
$$ \Vert \tilde{X} - X \Vert = \sqrt { \sum_{i\in d} \Vert X_i - \tilde{X}_i \Vert^2} \leq \frac{\sqrt{d}}{2^b} R = \gamma R,
$$
which is the desired claim. 
\end{proof}

Next, we prove Lemma~\ref{lemma:quant} from Section~\ref{sec:analysis} in the main body, which we restate below.

\begin{lemma} \label{lemma:quantApp} (\textbf{No Overflow and Quantization Error})  Suppose $f$ is $L$-smooth. The following are then true for all $t\geq 0$: (i) $\Vert i_t \Vert \leq R_t$; and (ii) $\Vert e_t \Vert \leq \gamma R_t,$ where $e_t=\nabla f(x_t)-g_t$.
\end{lemma}
\begin{proof}
Proof of part (i): We will prove this result via induction. For the base case of induction, we observe that at $t=0$, $i_0=\nabla f(x_0) - g_{-1}= \nabla f(x_0)$, since $g_{-1}=0$. Now since $R_0$ is chosen such that $\Vert \nabla f(x_0) \Vert \leq R_0,$ we conclude that $\Vert i_0 \Vert \leq R_0.$
 
Now suppose $\Vert i_t \Vert \leq R_t$ holds for $t=0, 1, \ldots, k,$ where $k$ is some positive integer. Our goal is to show that $\Vert i_{k+1} \Vert \leq R_{k+1}.$ To that end, we have
\begin{equation}
\begin{aligned}
\Vert i_{k+1} \Vert &\overset{(a)}= \Vert \nabla f(x_{k+1})- g_k \Vert\\
        &= \Vert \nabla f(x_{k+1}) - \nabla f(x_k) + \nabla f(x_k) - g_k \Vert \\
        & \overset{(b)} \leq L \Vert x_{k+1} - x_k \Vert + \Vert \nabla f(x_k) - g_k \Vert \\
        & \overset{(c)} \leq \alpha L \Vert g_k \Vert + \Vert \nabla f(x_k) - g_k \Vert. 
\end{aligned}
\label{eqn:ind1}
\end{equation}
In the above steps, (a) follows from the definition of the innovation $i_{k+1}$ (line 5 of \texttt{AQGD}); (b) follows from $L$-smoothness; and (c) is a consequence of the update rule of \texttt{AQGD} in Eq.~\eqref{eqn:AQGD}. Now from line 8 of \texttt{AQGD}, we have 
$$ g_k = g_{k-1} + \tilde{i}_k = \nabla f(x_k) + \tilde{i}_k - i_k,$$
where in the last step, we used $i_k = \nabla f(x_k) - g_{k-1}.$ 
We thus conclude that 
\begin{equation}
\Vert e_k \Vert = \Vert \nabla f(x_k) - g_k \Vert = \Vert i_k - \tilde{i}_k \Vert.
\label{eqn:ind2}
\end{equation}
From the induction hypothesis, we have $\Vert i_k \Vert \leq R_k.$ Since $\tilde{i}_k = Q_{b,R_k}(i_k),$ Lemma~\ref{lemma:SQ} then tells us that
$$ \Vert i_k - \tilde{i}_k \Vert \leq \gamma R_k.$$
Combining the above display with Eq.~\eqref{eqn:ind1} and Eq.~\eqref{eqn:ind2} yields:
$$ \Vert i_{k+1} \Vert \leq \gamma R_k + \alpha L \Vert g_k \Vert = R_{k+1},$$
where in the last step, we used the update rule for the range in Eq.~\eqref{eqn:Range_Update}. This concludes the induction step.

Proof of part (ii): In the above analysis, we have already shown that
$$\Vert e_t \Vert = \Vert \nabla f(x_t) - g_t \Vert = \Vert i_t - \tilde{i}_t \Vert.$$ Since from part (i), $\Vert i_t \Vert \leq R_t, \forall t \geq 0,$ we have $\Vert i_t - \tilde{i}_t \Vert \leq \gamma R_t$ based on Lemma~\ref{lemma:SQ}. We conclude that $\Vert e_t \Vert \leq \gamma R_t, \forall t\geq 0.$ This completes the proof. 
\end{proof}

Finally, we prove Lemma~\ref{lemma:range} from Section~\ref{sec:analysis} in the main body. 

\begin{lemma} (\textbf{Recursion for Dynamic Range}) \label{lemma:rangeApp} Suppose $f$ is $L$-smooth. If $\alpha$ is such that $\alpha L \leq 1$, then for all $t \geq 0$, we have:
\begin{equation}
R^2_{t+1} \leq 8 \gamma^2 R^2_t + 2 \alpha^2 L^2 \Vert \nabla f(x_t) \Vert^2. 
\end{equation}
\end{lemma}
\begin{proof}
From the update rule for the range in Eq.~\eqref{eqn:Range_Update}, we have
\begin{equation}
\begin{aligned}
R_{t+1} &= \gamma R_t + \alpha L \Vert g_t \Vert \\
& \leq \gamma R_t + \alpha L \left( \Vert g_t - \nabla f(x_t) \Vert + \Vert \nabla f(x_t) \Vert  \right)\\ 
& = \gamma R_t + \alpha L \Vert e_t \Vert +  \alpha L \Vert \nabla f(x_t) \Vert  \\
& \overset{(a)} \leq (1+\alpha L) \gamma R_t + \alpha L \Vert \nabla f(x_t) \Vert \\
& \overset{(b)} \leq 2 \gamma R_t + \alpha L \Vert \nabla f(x_t) \Vert,
\end{aligned}
\end{equation}
where for (a), we used Lemma~\ref{lemma:quantApp}, and for (b), we used the fact that $\alpha L \leq 1.$ Squaring both sides of the final inequality above, and using the elementary fact that $(a+b)^2 \leq 2 a^2 + 2 b^2, \forall a,b \in \mathbb{R},$ leads to the claim of the lemma. 
\end{proof}

With the above lemmas, We are in place to prove Theorem~\ref{thm:PL}.
\begin{proof} (\textbf{Proof of Theorem~\ref{thm:PL}}): Consider the following potential function:
\begin{equation}\label{eqn:potential function}
V_t \triangleq z_t + \alpha R^2_t,
\end{equation}
where $z_t = f(x_t)-f(x^*).$ Our goal is to establish that the above potential function decays to $0$ exponentially fast at a desired rate. To that end, we first recall that $L$-smoothness implies the following:
$$ f(y) \leq f(x) + \langle \nabla f(x), y-x \rangle + \frac{L}{2} \Vert y - x \Vert^2, \forall x, y \in \mathbb{R}^d.$$
Using the above display, observe
\begin{equation}
\begin{aligned}
f(x_{t+1}) &\leq f(x_t) + \langle \nabla f(x_t), x_{t+1}-x_t \rangle + \frac{L}{2} \Vert x_{t+1} - x_t \Vert^2\\ 
& \overset{(a)} \leq f(x_t) - \alpha \langle \nabla f(x_t), g_t \rangle + \frac{\alpha^2 L}{2} \Vert g_t \Vert^2\\
& \overset{(b)} \leq f(x_t) - \alpha \langle \nabla f(x_t), \nabla f(x_t) - e_t \rangle + \frac{\alpha^2 L}{2} \Vert \nabla f(x_t) - e_t \Vert^2\\
& \leq f(x_t) - \alpha (1-\alpha L) \Vert \nabla f(x_t) \Vert^2 + \alpha \langle \nabla f(x_t), e_t \rangle + \alpha^2 L \Vert e_t \Vert^2\\ 
& \overset{(c)} \leq f(x_t) - \alpha \left(\frac{1}{2} -\alpha L\right) \Vert \nabla f(x_t) \Vert^2 + \alpha \left(\frac{1}{2}+ \alpha L \right) \Vert e_t \Vert^2\\ 
& \overset{(d)} \leq f(x_t) - \alpha \left(\frac{1}{2} -\alpha L\right) \Vert \nabla f(x_t) \Vert^2 + \alpha \gamma^2 \left(\frac{1}{2}+ \alpha L \right) R^2_t.
\end{aligned}
\label{eqn:final1}
\end{equation}

In the above steps, for (a), we used Eq.~\eqref{eqn:AQGD}; for (b), we used $e_t = \nabla f(x_t) - g_t$; and for (c), we used the elementary fact that for any $a, b \in \mathbb{R}^d$, it holds that
$$ \langle a, b \rangle \leq \frac{1}{2} \Vert a \Vert^2 + \frac{1}{2} \Vert b \Vert^2.$$
Finally, (d) follows from part (ii) of Lemma~\ref{lemma:quantApp}. Using the final display in Eq.~\eqref{eqn:final1} in tandem with Lemma~\ref{lemma:rangeApp}, we then obtain
\begin{equation}
\begin{aligned}
V_{t+1} &= z_{t+1} + \alpha R^2_{t+1} \\
& \leq z_t -  \alpha \left(\frac{1}{2} -\alpha L\right) \Vert \nabla f(x_t) \Vert^2 + \alpha \gamma^2 \left(\frac{1}{2}+ \alpha L \right) R^2_t + \alpha R^2_{t+1}\\
& \leq  z_t - \frac{\alpha}{2} \left(1 - 2 \alpha L - 4 \alpha ^2 L^2 \right) \Vert \nabla f(x_t) \Vert^2 + \alpha \left(8 + \frac{1}{2} + \alpha L \right) \gamma^2 R^2_t\\
& \leq z_t - \frac{\alpha}{2} \left(1-3 \alpha L\right) \Vert \nabla f(x_t) \Vert^2  + 9 \alpha \gamma^2 R^2_t,
\end{aligned}
\end{equation}
where in the last step, we used $\alpha L \leq 1/2.$ Now suppose $\alpha L \leq 1/6.$ Then the above display in tandem with the gradient domination property (see Eq.~\eqref{eqn:PL}) yields:
\begin{equation}
V_{t+1} \leq \underbrace{\left(1 - \frac{\alpha \mu}{2} \right) z_t}_{T_1} + \underbrace{9 \alpha \gamma^2 R^2_t}_{T_2}.
\label{eqn:final2}
\end{equation}
Let us now make a couple of simple observations. Suppose we set $\alpha = \frac{1}{6L}$. Then, in the absence of the term $T_2$ above, it is easy to see that one can achieve a rate of the form $O(\exp(-t/\kappa))$, where $\kappa = L/\mu.$ To continue to preserve this rate in the presence of the term $T_2$ - which arises due to quantization errors - we need $T_2$ to decay faster than $T_1$. As such, we would like to have
$$ 9 \gamma^2 \leq \left(1-\frac{\alpha \mu}{2}\right) = \left(1-\frac{1}{12 \kappa}\right).   \hspace{5mm} \textrm{(Rate Preservation Condition)}$$
Recalling that $\gamma = \sqrt{d}/{2^b},$ a bit of simple algebra shows that 
$$ b \geq C \log \left( \frac{d\kappa}{\kappa-1} \right),$$
suffices to ensure the rate preservation condition above. Thus, suppose the bit-precision $b$ satisfies the condition in the above display. We then immediately obtain from Eq.~\eqref{eqn:final2} that
$$ V_{t+1} \leq \left(1 - \frac{1}{12 \kappa} \right) \left(z_t + \alpha R^2_t\right) = \left(1 - \frac{1}{12 \kappa} \right) V_t,$$
where in the last step, we used the definition of the potential function $V_t$ from Eq.~\eqref{eqn:Lyap}. Iterating the above inequality yields:
$$ V_t \leq \left(1 - \frac{1}{12 \kappa} \right)^t \left(f(x_0)-f(x^*) + \alpha R^2_0\right).$$
The result follows by noting that $f(x_t)-f(x^*) = z_t \leq V_t.$ 
\end{proof}

\section{Proofs Omitted in Section~\ref{sec:noisy AQDG}}\label{app:proofs for NAQGD}
\subsection{Proof of Theorem~\ref{thm:noisy gradient}}
Similarly to the proof of Theorem~\ref{thm:PL}, we prove Theorem~\ref{thm:noisy gradient} via an induction on $t\ge0$. However, since the innovation $i_t$ and the range $R_t$ of the quantizer map have been modified according to Eqs.~\eqref{eqn:innovation noisy case} and \eqref{eqn:Range_Update modified}, our analysis in this proof will be different from that for Theorem~\ref{thm:PL} . We first prove via the induction that $\Vert i_t\Vert\le R_t$, $\Vert e_t\Vert\le\gamma R_t+\varepsilon_t$ and $f(x_t)\le v/2$, $\forall t\ge0$, where $e_t=\nabla f(x_t)-g_t$.   

To prove the base case, we first recall that $f(x_0)\le v/4$ as assumed. Moreover, since $R_0$ is picked to satisfy $\Vert \widehat{\nabla f(x_0)}\Vert\le R_0$, we have $\Vert i_0\Vert=\Vert\widehat{\nabla f(x_0)}\Vert\le R_0$. In fact, note that 
\begin{equation}\label{eqn:upper bound on noisy gradient at x_0}
\begin{aligned}
\Vert\widehat{\nabla f(x_0)}\Vert&=\Vert \widehat{\nabla f(x_0)}-\nabla f(x_0)+\nabla f(x_0)\Vert\\
&\le\Vert\widehat{\nabla f(x_0)}-\nabla f(x_0)\Vert+\Vert\nabla f(x_0)\Vert\\
&\le\varepsilon_0+G,
\end{aligned}
\end{equation}
where the last inequality follows from the assumptions in Eq.~\eqref{eqn:upper bound on gradient noise} and  $\Vert\nabla f(x_0)\Vert\le G$ (since $x_0\in\calX$). Thus, we can choose $R_0=\varepsilon_0+G$. By the definition of Algorithm~\ref{algo:AQGD}, we get
\begin{equation}\label{eqn:upper bound on e_0 noisy case}
\begin{aligned}
\Vert e_0\Vert&=\Vert \nabla f(x_0)-g_0\Vert\\
&=\Vert\nabla f(x_0)-g_{-1}-\tilde{i}_0\Vert\\
&=\Vert \nabla f(x_0)-\widehat{\nabla f(x_0)}+i_0-\tilde{i}_0\Vert\\
&\le\Vert \nabla f(x_0)-\widehat{\nabla f(x_0)}\Vert+\Vert i_0-\tilde{i}_0\Vert\\
&\le \varepsilon_0+\gamma R_0,
\end{aligned}
\end{equation}
where the last inequality follows from the assumption in Eq.~\eqref{eqn:upper bound on gradient noise} and Lemma~\ref{lemma:SQ} via the fact that $\Vert i_0\Vert\le R_0$.

To prove the induction step, suppose the following induction hypotheses hold for $k=0,1,\dots,t$: (i) $\Vert i_k\Vert\le R_k$; and (ii) $f(x_k)\le v/2$. Using the same arguments as Eq.~\eqref{eqn:upper bound on e_0 noisy case} and the induction hypothesis $\Vert i_t\Vert\le R_t$, one can show that $\Vert e_t\Vert\le\varepsilon_t+\gamma R_t$. We now aim to show that $\Vert x_{t+1}-x_t\Vert\le D$ so that the local properties of $f(\cdot)$ can be applied. We begin by relating $x_{t+1}$ to $x_t$ as
\begin{equation}\label{eqn:x_t+1 and x_t relation}
\Vert x_{t+1}-x_t\Vert=\Vert x_t-\alpha g_t-x_t\Vert=\alpha\Vert g_t\Vert\le\alpha\Vert e_t\Vert+\alpha\Vert \nabla f(x_t)\Vert.
\end{equation}
From Eq.~\eqref{eqn:Range_Update modified}, we also have
\begin{equation}\label{eqn:upper bound on R_t noisy case}
\begin{aligned}
R_{t+1}&=\gamma R_t+\alpha L\Vert g_t\Vert+\varepsilon_t+\varepsilon_{t+1}\\
&\le\gamma R_t+\alpha L\Vert e_t\Vert+\alpha L\Vert\nabla f(x_t)\Vert+\varepsilon_t+\varepsilon_{t+1}\\
&\overset{(a)}\le\gamma(1+\alpha L) R_t+\alpha L\Vert\nabla f(x_t)\Vert+(1+\alpha L)\varepsilon_t+\varepsilon_{t+1}\\
&\overset{(b)}\le2\gamma R_t+\alpha LG+2\varepsilon_t+\varepsilon_{t+1},
\end{aligned}
\end{equation}
where (a) follows from the fact $\Vert e_t\Vert\le\varepsilon_t+\gamma R_t$ as we argued above, and (b) is due to the choice of $\alpha$ that satisfies $\alpha L\le 1$ and the gradient upper bound $\Vert \nabla f(x_t)\Vert\le G$ by the induction hypothesis $f(x_t)\le v/2$, i.e., $x_t\in\calX$. Setting $\gamma$ (i.e., $b$) in the same way as Theorem~\ref{thm:PL} (see our arguments in the proof of Theorem~\ref{thm:PL}) yields 
\begin{equation*}
9\gamma^2\le \left(1-\frac{\alpha\mu}{2}\right)\le 1,  
\end{equation*}
where the second inequality follows from the choice of $\alpha$; we conclude that  $0<\gamma^2\le 1/9$ (i.e., $0<\gamma\le1/3$). Unrolling the recursion in Eq.~\eqref{eqn:upper bound on R_t noisy case}, we obtain 
\begin{equation}\label{eqn:upper bound on R_t noisy case unrolled}
\begin{aligned}
R_t&\le(2\gamma)^tR_0+\sum_{k=1}^{t}(2\gamma)^{k-1}(\alpha LG+2\varepsilon_{t-k}+\varepsilon_{t+1-k})\\
&\le G+\varepsilon_0+\frac{\alpha LG+3\bar{\varepsilon}}{1-2\gamma}\\
&\overset{(a)}\le 4G+10\bar{\varepsilon}\\
&\overset{(b)}\le 14G
\end{aligned}
\end{equation}
where (a) follows from $0<\gamma\le1/3$ as we argued above and $\bar{\varepsilon}\triangleq\max_{t\ge0}\varepsilon_t$, and (b) follows from the assumption that $\varepsilon_t\le G$ for all $t\ge0$. Substituting th last relation in Eq.~\eqref{eqn:upper bound on R_t noisy case unrolled} into $\Vert e_t\Vert\le\varepsilon_t+\gamma R_t$ yields $\Vert e_t\Vert\le G+14\gamma G\le 6G$, where we again invoke the assumption $\varepsilon_t\le G$ and the fact $0<\gamma\le1/3$. Returning to Eq.~\eqref{eqn:x_t+1 and x_t relation}, we get
\begin{equation}\label{eqn:x_t+1 and x_t distance D}
\begin{aligned}
\Vert x_{t+1}-x_t\Vert&\le\alpha\Vert e_t\Vert+\alpha\Vert\nabla f(x_t)\Vert\\
&\overset{(a)}\le6\alpha G+\alpha G\\
&\overset{(b)}\le D,
\end{aligned}
\end{equation}
where (a) follows from the induction hypothesis $f(x_t)\le v/2$ which implies that $x_t\in\calX$ and $\Vert\nabla f(x_t)\Vert\le G$, and (b) follows from the choice of $\alpha$ that satisfies $\alpha\le D/(7G)$. It follows from Eq.~\eqref{eqn:x_t+1 and x_t distance D} that the local $(L,D)$-smoothness property can be applied to $x_{t+1}$ and $x_t$. We can now upper bound $\Vert i_{t+1}\Vert$ as
\begin{equation}
\begin{aligned}
\Vert i_{t+1}\Vert&=\Vert\widehat{\nabla f(x_{t+1})}-g_k\Vert\\
&\le\Vert\widehat{\nabla f(x_{t+1})}-\nabla f(x_{t+1})\Vert+\Vert\nabla f(x_{t+1})-\nabla f(x_{t})\Vert+\Vert\nabla f(x_t)-g_t\Vert\\
&\overset{(a)}\le\varepsilon_{t+1}+\Vert e_t\Vert+\Vert\nabla f(x_{t+1})-\nabla f(x_{t})\Vert\\
&\overset{(b)}\le\varepsilon_{t+1}+\varepsilon_t+\gamma R_t+L\Vert x_{t+1}-x_{t}\Vert\\
&\le\varepsilon_{t+1}+\varepsilon_t+\gamma R_t+L\Vert x_t-\alpha g_t-x_t\Vert\\
&\le\varepsilon_{t+1}+\varepsilon_t+\gamma R_t+\alpha L\Vert g_t\Vert\\
&\le\gamma R_t +\alpha L\Vert g_t\Vert+\varepsilon_{t+1}+\varepsilon_{t}\\
&=R_{t+1},
\end{aligned}
\end{equation}
where (a) follows from the assumption in Eq.~\eqref{eqn:upper bound on gradient noise}, and (b) follows from the fact $\Vert e_t\Vert\le\varepsilon_t+\gamma R_t$ as we argued above, and the $(L,D)$-smoothness of $f(\cdot)$. Thus, we have shown $\Vert i_{t+1}\Vert\le R_{t+1}$ for the induction step. 

To finish the induction step, it remains to show that $f(x_{t+1})\le v/2$. To this end, following the arguments to those for Eq.~\eqref{eqn:final1} and leveraging the $(L,D)$-smoothness of $f(\cdot)$ via Lemma~\ref{lemma:local smooth 2nd}, one can relate $f(x_{t+1})$ to $f(x_t)$ as
\begin{equation}\label{eqn:relate f(x_t+1) to f(x_t)}
\begin{aligned}
f(x_{t+1})&\le f(x_t)-\alpha\left(\frac{1}{2}-\alpha L\right)\Vert\nabla f(x_t)\Vert^2+\alpha\left(\frac{1}{2}+\alpha L\right)\Vert e_t\Vert^2\\
&\le f(x_t)-\frac{\alpha}{3}\Vert\nabla f(x_t)\Vert^2+\alpha(1+2\alpha L)(\gamma^2R_t^2+\varepsilon_t^2),
\end{aligned}
\end{equation}
where the second inequality follows from $\Vert e_t\Vert\le\gamma R_t+\varepsilon_t$ and the choice of $\alpha$ that satisfies $\alpha L\le1/6$. Let us denote $z_t=f(x_t)-f(x^*)$ and consider the following potential function:
\begin{equation}
\begin{aligned}
V_t=z_t+z_{t-1}+\alpha R_t^2
\end{aligned}
\end{equation}
To relate $V_{t+1}$ to $V_{t}$, we first show that
\begin{equation}\label{eqn:relate V_t+1 to V_t}
\begin{aligned}
V_{t+1}&=z_{t+1}+z_t+\alpha R_t^2\\
&\le z_t-\frac{\alpha}{3}\Vert\nabla f(x_t) \Vert^2+\alpha(1+2\alpha L)(\gamma^2R_t^2+\varepsilon_t^2)\\
&\qquad+z_{t-1}-\frac{\alpha}{3}\Vert\nabla f(x_{t-1})\Vert^2+\alpha(1+2\alpha L)(\gamma^2R_{t-1}^2+\varepsilon_{t-1}^2)+\alpha R_t^2\\
&\le z_t-\frac{\alpha}{3}\Vert\nabla f(x_t)\Vert^2+\alpha(\gamma^2+2\gamma^2\alpha L+1)R_t^2\\
&\quad+z_{t-1}-\frac{\alpha}{3}\Vert\nabla f(x_{t-1})\Vert^2+\alpha(\gamma^2+2\gamma^2\alpha L)R_{t-1}^2+\alpha(1+2\alpha L)(\varepsilon_t^2+\varepsilon_{t-1}^2),
\end{aligned}
\end{equation}
where the first inequality follows from Eq.~\eqref{eqn:relate f(x_t+1) to f(x_t)}. Note that by squaring both sides of the second last inequality in Eq.~\eqref{eqn:upper bound on R_t noisy case} and using the fact $(\sum_{i=1}^na_i^2)\le n\sum_{i=1}^na_i^2$, $\forall n\in\bbZ_{\ge1}$ and $\forall a_i\in\bbR$, we get
\begin{equation}\label{eqn:upper bound on R_t^2 noisy case}
\begin{aligned}
R_{t}^2\le4\gamma^2(1+\alpha L)^2R_{t-1}^2+4\alpha^2L^2\Vert\nabla f(x_{t-1})\Vert^2+4(1+\alpha L)^2\varepsilon_{t-1}^2+4\varepsilon_{t}^2.
\end{aligned}
\end{equation}
Substituting Eq.~\eqref{eqn:upper bound on R_t^2 noisy case} into the last inequality of Eq.~\eqref{eqn:relate V_t+1 to V_t}, using the local gradient-domination property for $x_t\in\calX$, one can show via some algebra that
\begin{equation}
\begin{aligned}
V_{t+1}&\le\left(1-\frac{2\alpha\mu}{3}\right)z_t+\left(1-2\left(\frac{1}{3}-4(\gamma^2+2\gamma^2\alpha L+1)\alpha^2L^2\right)\alpha\mu\right)z_{t-1}\\
&\qquad+\alpha\gamma^2\Big(4(\gamma^2+2\gamma^2\alpha L+1)(1+\alpha L)^2+1+2\alpha L\Big)R_{t-1}^2\\
&\qquad+\alpha\Big(4(\gamma^2+2\gamma^2\alpha L+1)(1+\alpha L)^2+1+2\alpha L\Big)\varepsilon_{t-1}^2\\
&\qquad+\alpha\Big(4(\gamma^2+2\gamma^2\alpha L+1)+1+2\alpha L\Big)\varepsilon_t^2\\
&\overset{(a)}\le\left(1-\frac{2\alpha\mu}{3}\right)z_t+\left(1-\frac{2\alpha\mu}{5}\right)z_{t-1}+8\alpha\gamma^2R_{t-1}^2+8\alpha\varepsilon_{t-1}^2+6\alpha\varepsilon_t^2\\
&\overset{(b)}\le\left(1-\frac{2\alpha\mu}{5}\right)(z_t+z_{t-1}+\alpha R_{t-1}^2)+8\alpha\varepsilon_{t-1}^2+6\alpha\varepsilon_{t}^2\\
&\overset{(c)}=\left(1-\frac{2\alpha\mu}{5}\right)V_{t}+8\alpha\varepsilon_{t-1}^2+6\alpha\varepsilon_t^2,
\end{aligned}
\end{equation}
where (a) uses the choice of $\alpha$ such that $\alpha L\le1/6$ and the fact $\gamma^2\le 1/9$, (b) follows from $8\gamma^2\le(1-2\alpha\mu/5)$ again due to the facts $\alpha L\le1/6$ and $\gamma^2\le1/9$, and (c) follows from the definition of $V_t$. To proceed, we split our analysis into two cases. First, supposing $8\varepsilon_{t-1}^2+6\varepsilon_t^2\le V_t\mu/15$, we have
\begin{equation}
\begin{aligned}
V_{t+1}&\le\left(1-\frac{2\alpha\mu}{5}\right)V_t+\frac{V_t\mu}{15}=\left(1-\frac{\alpha\mu}{3}\right)V_t.
\end{aligned}
\end{equation}
Second, supposing $8\varepsilon_{t-1}^2+6\varepsilon_t^2>V_t\mu/15$, we have
\begin{equation}
\begin{aligned}
V_{t+1}&\le\left(\left(1-\frac{2\alpha\mu}{5}\right)\frac{15}{\mu}+\alpha\right)(8\varepsilon_{t-1}^2+6\varepsilon_{t}^2)\\
&\le\frac{15}{\mu}(8\varepsilon_{t-1}^2+6\varepsilon_t^2).
\end{aligned}
\end{equation}
It follows from the above two cases that 
\begin{equation}
V_{t+1}\le\max\left\{\left(1-\frac{\alpha\mu}{3}\right)V_t,\frac{15}{\mu}(8\varepsilon_{t-1}^2+6\varepsilon_t^2)\right\}.
\end{equation}
Unrolling the above relation, one can show that
\begin{equation}
V_{t+1}\le\max\left\{\left(1-\frac{\alpha\mu}{3}\right)^{t+1}V_0,\frac{15\bar{\varepsilon}_t^2}{\mu}\right\},
\end{equation}
where $\bar{\varepsilon}_t$ is given by Eq.~\eqref{eqn:def of bar epsilon_t}. Also recalling the definition of $V_t$ and noting that $V_0=z_0+z_{-1}+\alpha R_{-1}^2$ with $z_{-1}\triangleq0$ and $R_{-1}\triangleq0$, we further obtain
\begin{equation}\label{eqn:distance to f(x^*)}
f(x_{t+1})-f(x^*)\le\max\left\{\left(1-\frac{\alpha\mu}{3}\right)^{t+1}(f(x_0)-f(x^*)),\frac{15\bar{\varepsilon}_t^2}{\mu}\right\}.
\end{equation}
Invoking the assumptions that $f(x_0)\le v/4$ (and thus $f(x^*)\le v/4$) and $15\bar{\varepsilon}_t^2/\mu\le v/4$, we get $f(x_{t+1})\le v/2$, completing the induction step. Note that we have also proved the desired convergence result.$\hfill\qed$

\subsection{Proof of Theorem~\ref{thm:noisy gradient for LQR}}
We begin by introducing some notations and definitions that will be used in this proof. Define an auxiliary function
\begin{equation}\label{eqn:smoothed version of J}
J^r(K)=\E[J(K+rU)],\forall K\in\calK,
\end{equation}
where $U$ is uniformly distributed over the set $\{U:\Vert U\Vert_F=1\}$.
In addition, introduce the following probabilistic event for all $t\in\{0,\dots,T-1\}$ and all $i\in[\ell_t]$:
\begin{equation}\label{eqn:event E_t}
\calE_t^i=\Big\{\Vert w_{t,k}^i\Vert\le\sqrt{5\trace(\Sigma_w)\log\frac{3N_t\ell_tT}{\delta}},\forall k\in\{0,\dots,N_t-1\}\Big\}.
\end{equation}

\begin{lemma}
\label{lemma:E_t holds with high probability}
The event $\calE_t^i$ defined in Eq.~\eqref{eqn:event E_t} holds with probability  $\bbP(\calE_t^i)\ge1-\delta/(3\ell_tT)$. 
\end{lemma}
\begin{proof}
By Lemma~\ref{lemma:gaussian}, for any $t\in\{0,\dots,T-1\}$ and any $k\in\{0,\dots,N_t-1\}$, it holds with probability at least $1-\delta/(3N_t\ell_tT)$ that $\Vert w_{t,k}^i\Vert\le\sqrt{5\trace(\Sigma_w)\log\frac{3N_t \ell_tT}{\delta}}$. Taking a union bound over all $k\in\{0,\dots,N_t-1\}$ completes the proof of the lemma.
\end{proof}

\begin{claim}\label{claim:induction}
Consider any $t\in\{0,\dots,T\}$. Then, (i) $J(K_{t^{\prime}})\le J/2$ for all $t^{\prime}\le t$ hold with probability at least $1-t\delta/T$; and (ii)  $\Vert\widehat{\nabla J(K_{t^{\prime}})}-\nabla J(K_{t^{\prime}})\Vert_F\le\varepsilon_{t^{\prime}}$ for all $t^{\prime}\le t$ hold with probability at least $1-(t+1)\delta/T$, where $\varepsilon_{t^{\prime}}$ is given in Eq.~\eqref{eqn:epsilon_t for LQR}.
\end{claim}
\begin{proof}
The base step in the induction holds since we assumed that $J(K_0)\le J/4$. Next, consider a given $t\in\{0,\dots,T-1\}$ and suppose $J(K_t)\le J/2$ for all $t^{\prime}\le t$ hold with probability at least $1-t\delta/T$. We will show that $J(K_{t+1})\le J/2$ holds with probability at least $1-(t+1)\delta/T$. Leveraging $J^r(\cdot)$ defined in Eq.~\eqref{eqn:smoothed version of J}, we first decompose the error of $\widehat{\nabla J(K_t)}$ as
\begin{align}
\Vert \widehat{\nabla J(K_t)}-\nabla J(K_t)\Vert_F&\le\Vert\widehat{\nabla J(K_t)}-\nabla J^r(K_t)\Vert_F+\Vert\nabla J^r(K_t)-\nabla J(K_t)\Vert_F,\label{eqn:noisy gradient decompose}
\end{align}
and we will upper bound the two terms on the right-hand side of the above relation separately in the sequel. To upper bound $\Vert\nabla J^r(K_t)-\nabla J(K_t)\Vert_F$, we show that 
\begin{equation}\label{eqn:upper bound 1}
\begin{aligned}
\Vert\nabla J^r(K_t)-\nabla J(K_t)\Vert_F&=\Vert\nabla\E_U[J(K_t+rU)]-\nabla J(K_t)\Vert_F\\
&\le\Vert\E_U[\nabla J(K_t+rU)-\nabla J(K_t)]\Vert_F\\
&\overset{(a)}\le\E_U\big[\Vert\nabla J(K_t+rU)-\nabla J(K_t)\Vert_F\big]\\
&\overset{(b)}\le Lr,
\end{aligned}
\end{equation}
where (a) follows from Jensen's inequality and the convexity of $\Vert\cdot\Vert$. To obtain (b), we first used the fact $K_t\in\calK$ from the induction hypothesis, and then invoked the local $(L,D)$-smoothness of $J(\cdot)$ given by Lemma~\ref{lemma:properties of J(K)}(d) based on the choice of $r$ that satisfies $r\Vert U\Vert_F\le D$, where recall that $\Vert U\Vert_F=1$. We next upper bound $\Vert\widehat{\nabla J(K_t)}-\nabla J^r(K_t)\Vert_F$. Denoting $c_{t,k}^i=x_{t,k}^{i\top}Qx_{t,k}^i+u_{t,k}^{i\top}Ru_{t,k}^i$, we see from Algorithm~\ref{algo:gradient estimate} that $\widehat{\nabla J(K_t)}=\frac{mn}{r\ell_t N_t}\sum_{i=1}^{\ell_t}\sum_{k=0}^{N_t-1}c_{t,k}^iU_t^i$. We may further decompose $\Vert\widehat{\nabla J(K_t)}-\nabla J^r(K_t)\Vert_F$ as 
\begin{align}\nonumber
\Vert \widehat{\nabla J(K_t)}-\nabla J^r(K_t)\Vert_F&\le\Big\Vert\frac{1}{\ell_tN_t}\sum_{i=1}^{\ell_t}\sum_{k=0}^{N_t-1}\frac{mn}{r}\big(c_{t,k}^i-J(K_t+rU_t^i)\big)U_t^i\Big\Vert_F\\\nonumber
&\qquad\qquad+\Big\Vert\frac{1}{\ell_tN_t}\sum_{i=1}^{\ell_t}\sum_{k=1}^{N_t-1}\big(\frac{mn}{r}J(K_t+rU_t^i)U_t^i-\nabla J^r(K_t)\big)\Big\Vert_F\\\nonumber
&\le\underbrace{\frac{1}{\ell_t}\Big\Vert\sum_{i=1}^{\ell_t}\sum_{k=0}^{N_t-1}\frac{mn}{N_tr}\big(c_{t,k}^i-J(K_t+rU_t^i)\big)U_t^i\Big\Vert_F}_{E_t^1}\\
&\qquad\qquad+\underbrace{\frac{1}{\ell_t}\Big\Vert\sum_{i=1}^{\ell_t}\big(\frac{mn}{r}J(K_t+rU_t^i)U_t^i-\nabla J^r(K_t)\big)\Big\Vert_F}_{E_t^2}.\label{eqn:noisy gradient decompose 2}
\end{align}
To upper bound $E_t^2$ in \eqref{eqn:noisy gradient decompose 2}, we first note that $K_t\in\calK$ (since $J(K_t)\le J/2$ by the induction hypothesis) and obtain that for any $i\in[\ell_t]$, 
\begin{equation*}
\begin{aligned}
J(K_t+rU_t^i)-J(K_t)\overset{(a)}\le\bar{G}\Vert rU_t^i\Vert_F\overset{(b)}\le \frac{J}{2},
\end{aligned}
\end{equation*}
where (a) follows from the $(D,\bar{G})$-local Lipschitz of $J(\cdot)$ in Lemma~\ref{lemma:properties of J(K)}(d) via the choice of $r$ that satisfies $r\Vert U_t^i\Vert_F=r\le D$, and (b) follows from the choice of $r$ that satisfies $r\le J/(2\bar{G})$. Hence, we have $J(K_t+rU_t^i)\le J$ and $(K_t+rU_t^i)\in\calK$. Now, recalling from  Remark~\ref{remark:random smoothing} that $U_t^i$ is equal in distribution to a random matrix selected uniformly from the set $\{U:\Vert U\Vert_F=1\}$, one can obtain from \cite[Lemma~2.1]{flaxman2005online} that 
\begin{equation*}
\nabla J^r(K_t)=\E\Big[\frac{mn}{r}J(K_t+rU_t^i)U_t^i\big|K_t\Big],\forall i\in[\ell_t].
\end{equation*}
Denoting $X_t^i=\frac{mn}{r}J(K_t+rU_t^i)U_t^i-\nabla J^r(K_t)$, we obtain from our arguments above that 
\begin{equation*}
\Vert X_t^i\Vert_F\le\big\Vert\frac{mn}{r}J(K_t+rU_t^i)U_t^i\big\Vert_F+\Vert\nabla J^r(K_t)\Vert_F\le\frac{2mnJ}{r},
\end{equation*}
and $\E[X_t^i|K_t]=0$ for all $i\in[\ell_t]$. Thus, we can apply Lemma~\ref{lemma:azuma} and obtain that with probability at least $1-\delta/(3T)$,
\begin{equation}\label{eqn:upper bound on E_t^2}
\begin{aligned}
E_t^2=\frac{1}{\ell_t}\big\Vert\sum_{i=1}^{\ell_t}X_t^i\big\Vert\le\frac{2mnJ}{\ell_t r}\sqrt{2\ell_t\log\frac{45T}{\delta}}.
\end{aligned}
\end{equation}

To upper bound $E_1$ in \eqref{eqn:noisy gradient decompose 2}, we first consider 
\begin{equation*}
\widetilde{E}_t^1=\frac{1}{\ell_t}\Big\Vert\sum_{i=1}^{\ell_t}\sum_{k=0}^{N_t-1}\frac{mn}{N_tr}\big(\mathbb{1}\{\calE_t^i\}c_{t,k}^i-J(K_t+rU_t^i)\big)U_t^i\Big\Vert_F.
\end{equation*}
Note that $\widetilde{E}_t^1=E_t^1$ holds with probability $\bbP(\calE_t)$, where $\calE_t\triangleq\cap_{i\in[\ell_t]}\calE_t^i$. Using Lemma~\ref{lemma:E_t holds with high probability}, we have via a union bound that $\bbP(\calE_t)\ge1-\delta/(3T)$. Thus, an upper bound on $\widetilde{E}_t^1$ will give an upper bound on $E_t^1$ that holds with probability at least $1-\delta/(3T)$. To this end, we upper bound $\widetilde{E}_t^1$ and again decompose 
\begin{multline}\label{eqn:upper bound on tilde E_t^1}
\widetilde{E}_t^1\le\frac{1}{\ell_t}\Big\Vert\sum_{i=1}^{\ell_t}\underbrace{\sum_{k=0}^{N_t-1}\frac{mn}{N_tr}\big(\mathbb{1}\{\calE_t\}c_{t,k}^i-\E[\mathbb{1}\{\calE_t\}c_{t,k}^i|K_t,U_t^i]\big)U_t^i}_{Y_t^i}\Big\Vert_F\\+\frac{1}{\ell_t}\sum_{i=1}^{\ell_t}\Big|\underbrace{\sum_{k=0}^{N_t-1}\frac{mn}{N_tr}\big(\E[\mathbb{1}\{\calE_t\}c_{t,k}^i|K_t,U_t^i]-J(K_t+rU_t^i)\big)}_{Z_t^i}\Big|,
\end{multline}
where we used the fact that $\Vert U_t^i\Vert_F=1$. Note that $\E[Y_t^i|K_t,U_t^i]=0$ for all $i\in[\ell_t]$. Moreover, we have that under the event $\calE_t$,
\begin{equation*}
\begin{aligned}
c_{t,k}^i&=x_{t,k}^{i\top}\big(Q+(K_t+rU_{t}^i)^{\top}R(K_t+rU_t^i)\big)x_{t,k}^i\\
&\le\Vert Q+(K_t+rU_{t}^i)^{\top}R(K_t+rU_t^i)\Vert\Vert x_{t,k}^i\Vert^2\\
&\overset{(a)}\le(1+\zeta^2)\beta_1\Vert x_{t,k}^i\Vert^2\\
&\overset{(b)}\le\frac{(1+\zeta^2)\beta_1\zeta^2}{(1-\eta)^2}\max_{0\le k\le N_t-1}\Vert w_{t,k}^i\Vert^2\\
&\overset{(c)}\le\frac{5(1+\zeta^2)\beta_1\zeta^2}{(1-\eta)^2}\trace(\Sigma_w)\log\frac{3N_t\ell_tT}{\delta},\forall i\in[\ell_t]
\end{aligned}
\end{equation*}
where (a) follows from the fact $(K_t+rU_t^i)\in\calK$ and Lemma~\ref{lemma:properties of J(K)}(b), and (b) follows from Lemma~\ref{lemma:upper bound on state}, and (c) follows from Eq.~\eqref{eqn:event E_t}. Thus, we get from the definition of $Y_t^i$ that  
\begin{equation}
\begin{aligned}
\Vert Y_t^i\Vert_F&\le\frac{2mn}{N_tr}\sum_{k=0}^{N_t-1}\mathbb{1}\{\calE_t\}c_{t,k}^i\Vert U_{t}^i\Vert_F\\
&\le\frac{10mn(1+\zeta^2)\beta_1\zeta^2}{r(1-\eta)^2}\trace(\Sigma_w)\log\frac{3N_t\ell_tT}{\delta}.
\end{aligned}
\end{equation}
We can now apply Lemma~\ref{lemma:azuma} and obtain with probability at least $1-\delta/(3T)$,
\begin{equation}\label{eqn:upper bound on sum Y_t^i}
\frac{1}{\ell_t}\big\Vert\sum_{i=1}^{\ell_t}Y_t^i\Vert_F\le\frac{10mn(1+\zeta^2)\beta_1\zeta^2}{\ell_t(1-\eta)^2}\trace(\Sigma_w)\log\frac{3N_t\ell_tT}{\delta}\sqrt{2\ell_t\log\frac{45T}{\delta}}.
\end{equation}

To upper bound $Z_t^i$, we first observe that $\mathbb{1}\{\calE_t\}c_{t,k}^i$ is equal (with probability one) to the cost corresponding to the system $x_{t,s+1}^i=(A+B(K_t+rU_t^i))x_{t,s}^i+\widetilde{w}_{t,s}^i$ for $s=0,1,\dots$, where $\widetilde{w}_{t,s}^i=\mathbb{1}\{\calE_t\}w_{t,s}^i$. We then write 
\begin{equation}\label{eqn:upper bound on Z_t^i}
\begin{aligned}
|Z_t^i|&=\frac{mn}{N_tr}\sum_{k=0}^{N_t-1}\Big|\E[\mathbb{1}\{\calE_t\}c_{t,k}^i|K_t,U_t^i]-\widetilde{J}(K_t+rU_t^i)+\widetilde{J}(K_t+rU_t^i)-J(K_t+rU_t^i)\Big|\\
&\le\underbrace{\frac{mn}{N_tr}\sum_{k=0}^{N_t-1}\Big|\E[\mathbb{1}\{\calE_t\}c_{t,k}^i|K_t,U_t^i]-\widetilde{J}(K_t+rU_t^i)\Big|}_{Z_{t,1}^i}+\underbrace{\frac{mn}{N_tr}\sum_{k=0}^{N_t-1}\Big|\widetilde{J}(K_t+rU_t^i)-J(K_t+rU_t^i)\Big|}_{Z_{t,2}^i},
\end{aligned}
\end{equation}
where $\widetilde{J}(\cdot)$ is the cost defined as \eqref{eqn:LQR obj J(K)} when the disturbance $w_{t,k}^i$ is replaced by $\widetilde{w}_{t,k}^i$ described above. Similarly to Eq.~\eqref{eqn:expression for J(K)}, we have $\widetilde{J}(K)=\trace(P_K\widetilde{\Sigma}_w)$ for any $K\in\calK$, where $P_K$ is the solution to the Ricatti equation in Eq.~\eqref{eqn:DARE} and $\widetilde{\Sigma}_w=\E[\widetilde{w}_k\widetilde{w}_k^{\top}]$ is the covariance of the disturbance $\widetilde{w}_k=\mathbb{1}\{\calE_t\}w_k$ \cite{bertsekas2015dynamic}. Using Lemma~\ref{lemma:cost relation}, we get 
\begin{equation*}
\begin{aligned}
Z_{t,1}^1&\le\frac{mn}{N_tr}\frac{2\beta_1\zeta^6}{(1-\eta)^3}\max_{0\le k\le N_t-1}\Vert \widetilde{w}_{t,k}^i\Vert^2\\
&\le\frac{10mn\beta_1\zeta^6}{N_tr(1-\eta)^3}\trace(\Sigma_w)\log\frac{3N_t\ell_tT}{\delta},
\end{aligned}
\end{equation*}
where we recalled the definition of $\calE_t^i$ in Eq.~\eqref{eqn:event E_t}. Moreover, denoting $P_{K+rU_t^i}$ as the solution to the Ricatti equation in Eq.~\eqref{eqn:DARE} associated with $K_t+rU_t^i$, we have 
\begin{equation*}
\begin{aligned}
Z_{t,2}^i&=\frac{mn}{r}\Big|\widetilde{J}(K_t+rU_t^i)-J(K_t+rU_t^i)\Big|\\
&=\frac{mn}{r}\Big|\trace\big(P_{K_t+rU_t^i}(\widetilde{\Sigma}_w-\Sigma_w)\big)\Big|\\
&\overset{(a)}=\frac{mn}{r}\trace\big(P_{K_t+rU_t^i}(\E\big[\mathbb{1}\{(\calE_t^{i})^c\}w_{t,k}^iw_{t,k}^{i\top}\big])\big)\\
&\overset{(b)}\le\frac{mn}{r}\big\Vert P_{K_t+rU_t^i}\big\Vert\trace\big(\E\big[\mathbb{1}\{(\calE_t^{i})^c\}w_{t,k}^iw_{t,k}^{i\top}\big]\big)\\
&\overset{(c)}\le\frac{2mn\beta_1\zeta^4}{r(1-\eta)}\E\big[\mathbb{1}\{(\calE_t^{i})^c\}\Vert w_{t,k}^i\Vert^2\big]\\
&\overset{(d)}\le\frac{10mn\beta_1\zeta^4}{3TN_t r(1-\eta)}\trace(\Sigma_w)\log\frac{27TN_t}{\delta},
\end{aligned}
\end{equation*}
where (a) follows from the fact that $\Sigma_w-\widetilde{\Sigma}_w=\E[\mathbb{1}\{(\calE_t^{i})^c\}w_{t,k}^{i}w_{t,k}^{i\top}]\succeq0$, (b) follows from standard trace inequality \cite{wang1986trace}, (c) follows from Lemma~\ref{lemma:upper bound on state} via the fact that $(K_t+rU_t^i)\in\calK$ as we argued above, and (d) follows from Lemma~\ref{lemma:gaussian} via the fact $\bbP((\calE_t^i)^c)\le\delta/(3N_tT)$ by Lemma~\ref{lemma:E_t holds with high probability}. Hence, we have shown an upper bound on $|Z_t^i|$ as per~\eqref{eqn:upper bound on Z_t^i}. 

Now, recalling that \eqref{eqn:upper bound on sum Y_t^i} holds with probability at least $1-\delta/(3T)$, we can use \eqref{eqn:upper bound on sum Y_t^i} and \eqref{eqn:upper bound on Z_t^i} in \eqref{eqn:upper bound on tilde E_t^1} to obtain an upper bound on $\widetilde{E}_t^1$ that holds with probability at least $1-\delta/(3T)$. Since $\widetilde{E}_t^1=E_t^1$ holds with probability $\bbP(\calE_t)\ge1-\delta/(3T)$ as we argued above, a union bound implies an upper bound on $E_t^1$ as per~\eqref{eqn:upper bound on tilde E_t^1} that holds with probability at least $1-2\delta/(3T)$. Additionally, recalling that we have shown that the upper bound on $E_t^2$ in \eqref{eqn:upper bound on E_t^2} holds with probability at least $1-\delta/(3T)$, we can further apply a union bound to combine the upper bound on $\widetilde{E}_t^1$ in \eqref{eqn:upper bound on tilde E_t^1} and the upper bound on $E_t^2$ in \eqref{eqn:upper bound on E_t^2}, which gives the upper bound $\Vert \widehat{\nabla J(K_t)}-\nabla J^r(K_t)\Vert_F\le\varepsilon_t$ as per~\eqref{eqn:noisy gradient decompose 2} that holds with probability as least $1-\delta/T$, where $\varepsilon_t$ is given by Eq.~\eqref{eqn:epsilon_t for LQR}. Finally, since all the arguments above for the induction step rely on the induction hypothesis $J(K_t)\le J/2$ that holds with probability at least $1-t\delta/T$, another union bound shows that the upper bound  $\Vert \widehat{\nabla J(K_t)}-\nabla J^r(K_t)\Vert_F\le\varepsilon_t$ holds with probability at least $1-(t+1)\delta/T$. By our choices of $N_t$ and $\ell_t$ given in Theorem~\ref{thm:noisy gradient for LQR}, we have that $15\bar{\varepsilon}_t^2/\mu\le J/4$, where $\mu$ is given in Lemma~\ref{lemma:properties of J(K)}(e) and $\bar{\varepsilon}_t$ is defined as Eq.~\eqref{eqn:def of bar epsilon_t} using $\varepsilon_t$ in Eq.~\eqref{eqn:epsilon_t for LQR}. Hence, one can now use similar arguments to those in the proof of Theorem~\ref{thm:noisy gradient} leading up to \eqref{eqn:distance to f(x^*)} to show that $J(K_{t+1})\le J/2$ holds with probability at least $1-(t+1)\delta/T$, completing the induction step of the proof of part~(i) in Claim~\ref{claim:induction}. Note that part~(ii) of Claim~\ref{claim:induction} was readily proved by our above arguments.
\end{proof}

\begin{proof}
({\bf Proof of Theorem~\ref{thm:noisy gradient}}) We see from Claim~\ref{claim:induction} that the following hold with probability at least $1-\delta$: $J(K_t)\le v/2$ for all $t\in\{0,\dots,T\}$, and $\Vert\widehat{\nabla J(K_t)}-\nabla J(K_t)\Vert_F\le\varepsilon_t$ for all $t\in\{0,\dots,T-1\}$. Based on the choices of $N_t,\ell_t$ such that $15\bar{\varepsilon}_t^2/\mu\le J/4$ as we argued above, the results in Theorem~\ref{thm:noisy gradient for LQR} follow directly from the arguments in the proof of Theorem~\ref{thm:noisy gradient}
\end{proof}

\section{Auxiliary Lemmas}\label{app:aux}
\begin{lemma}\label{lemma:local smooth 2nd}
Let $f:\mathbb{R}^d\to\mathbb{R}$ be locally $(L,D)$-smooth over $\mathcal{X}\subseteq\mathbb{R}^d$. Then, for any $x\in\mathcal{X}$ and any $y\in\mathbb{R}^d$ with $\Vert y-x \Vert\le D$, 
\begin{align}
f(y) \leq f(x) + \langle \nabla f(x), y-x \rangle + \frac{L}{2} \Vert y - x \Vert^2.
\end{align}
\end{lemma}
\begin{proof}
The proof follows from adapting the standard result in e.g. \cite[Lemma~3.4]{bubeck2015convex}. Fix any $x \in \mc{X}$ and $y \in \mathbb{R}^d$ such that $\Vert y-x \Vert \leq D$. We may write $f(y)-f(x)$ as an integral and obtain
\begin{align*}
&f(y)-f(x)-\langle \nabla f(x), y-x \rangle\\
=&\int_{0}^1\langle\nabla f(x+t(y-x)),y-x\rangle dt-\langle\nabla f(x),y-x\rangle \\
\overset{(a)}\le&\int_0^1\Vert \nabla f(x+t(y-x))-\nabla f(x))\Vert \Vert y-x \Vert dt\\
\overset{(b)}\le&\int_{0}^1 Lt \Vert y-x \Vert^2dt\\
=&\frac{L}{2}\Vert y-x \Vert^2,
\end{align*}
where (a) follows from the Cauchy-Schwarz inequality, and (b) follows from Definition~\ref{def:local smooth} via the fact that $\Vert x+t(y-x)-x \Vert \le D$ for all $t\in[0,1]$. 
\end{proof}

\begin{lemma}\label{lemma:upper bound on state}
Let $K\in\calK$ with $\calK$ defined in Eq.~\eqref{eqn:feasible set for J(K)}. Then, for any $k\ge1$, the state of system~\eqref{eqn:LTI} given by $x_{s+1}=Ax_s+Bu_s+w_s$ for $s=0,\dots,k-1$ with $u_s=Kx_s$ and $x_0=0$ satisfies that 
\begin{equation}
\Vert x_{k}\Vert\le\frac{\zeta}{1-\eta}\max_{0\le s\le k-1}\Vert w_s\Vert,
\end{equation}
where $\zeta$ and $\eta$ are given in Lemma~\ref{lemma:properties of J(K)}. Moreover, the solution to the Ricatti equation \eqref{eqn:DARE}, i.e., $P_K\in\mathbb{S}_{++}^n$, satisfies that $$\Vert P_K\Vert\le\frac{2\beta_1\zeta^4}{1-\eta}.$$
\end{lemma}
\begin{proof}
First, one can obtain from \eqref{eqn:LTI} that
\begin{equation}
x_k=(A+BK_t)^{k}x_0+\sum_{s=0}^{k-1}(A+BK_t)^{k-(s+1)}w_{s}.
\end{equation}
Recalling from  Lemma~\ref{lemma:properties of J(K)}(b) that $\Vert(A+BK_t)^k\Vert\le\zeta\eta^k$, we further obtain 
\begin{equation}
\Vert x_{k}\Vert\le\zeta\sum_{s=0}^{k-1}\eta^k\Vert w_s\Vert\le\frac{\zeta}{1-\eta}\max_{0\le s\le k-1}\Vert w_s\Vert.
\end{equation}
The upper bound on $\Vert P_K\Vert$ is proved in the proof of \cite[Lemma~40]{cassel2020logarithmic}.
\end{proof}

\begin{lemma}\label{lemma:cost relation}
Let $K\in\calK$ with $\calK$ defined in Eq.~\eqref{eqn:feasible set for J(K)}. Consider the state of system~\eqref{eqn:LTI} given by $x_{s+1}=(A+BK)x_s+w_s$ for $s=0,\dots,k-1$ with $x_0=0$. Then, for any $k\ge1$,
\begin{equation*}
\Big|\sum_{s=1}^k\E\big[x_s^{\top}(Q+K^{\top}RK)x_s\big]-kJ(K)\Big|\le\frac{2\beta_1\zeta^6}{(1-\eta)^3}\E\Big[\max_{0\le s\le k}\Vert w_s\Vert^2\Big],
\end{equation*}
where $J(\cdot)$ is the cost of $K$ given in \eqref{eqn:LQR obj J(K)}, $\beta_1$ is given in \eqref{eqn:parameters in J(K)} and $\eta,\zeta$ are given in Lemma~\ref{lemma:properties of J(K)}.
\end{lemma}
\begin{proof}
Following the arguments in the proof of \cite[Lemma~40]{cassel2020logarithmic}, we have 
\begin{equation*}
\begin{aligned}
\Big|\sum_{s=1}^k\E\big[x_s^{\top}(Q+K^{\top}RK)x_s\big]-kJ(K)\Big|=\E\big[x_{k+1}^{\top}P_Kx_{k+1}\big]\le\Vert P_K\Vert\E[x_{k+1}^{\top}x_{k+1}],
\end{aligned}
\end{equation*}
where $P_K\in\bbS_{++}^n$ is the positive definite to the Ricatti equation given by Eq.~\eqref{eqn:DARE}.  Now, using the result of Lemma~\ref{lemma:upper bound on state}, we finish the proof of Lemma~\ref{lemma:cost relation}.
\end{proof}

\begin{lemma}\label{lemma:gaussian}
Let $w\sim\calN(0,\Sigma_w)$ with $\Sigma_w\in\mathbb{S}_{++}^n$. (a) For any $0<\delta<1/e$, the following holds with probability at least $1-\delta$:
\begin{equation*}
\Vert w\Vert\le\sqrt{5\trace(\Sigma_w)\log\frac{1}{\delta}}.
\end{equation*}
(b) Let $0<\delta^{\prime}<1$ and let $\calE$ be a probabilistic event that holds with probability $\bbP(\calE)\le \delta^{\prime}$. Then,
\begin{equation*}
\E\big[\mathbb{1}\{\calE\}\Vert w\Vert^2\big]\le 5\trace(\Sigma_w)\delta^{\prime}\log\frac{9}{\delta^{\prime}}.
\end{equation*}
\end{lemma}
\begin{proof}
Part~(a) is \cite[Lemma~14]{cassel2021online}. Using similar arguments to those in the proof of \cite[Lemma~35]{cassel2020logarithmic}, one can prove part~(b).
\end{proof}

\begin{lemma}\label{lemma:azuma}\textbf{(Vector Azuma Inequality)} 
Given a real Euclidean space $E$, let $\Vert X_k\Vert_{k\ge1}$ be a martingale difference sequence adapted to a filtration $\{\calF_k\}_{k\ge0}$, i.e., $X_k\in E$ is $\calF_k$-measurable and $\E[X_k|\calF_{k-1}]=0$ for all $k\ge1$. Suppose $\Vert X_k\Vert\le b$ for all $k\ge0$. Then, for any $s\ge1$ and any $0<\delta<\frac{1}{2}e^{-2}$,
\begin{equation*}
\big\Vert\sum_{k=1}^sX_k\big\Vert\le b\sqrt{2s\log\frac{15}{\delta}}.
\end{equation*}
\end{lemma}
\begin{proof}
This lemma is a direct consequence of \cite[Theorem~1.8]{hayes2005large}.
\end{proof}

\end{document}